\documentclass{article}

\usepackage[english]{babel}
\usepackage[latin1]{inputenc}
\usepackage{amssymb}
\usepackage{amsmath}
\usepackage{amsthm}
\usepackage{color}
\usepackage{mathrsfs}
\usepackage{enumerate}
\usepackage{cases}
\usepackage{changes}
\usepackage{appendix}
\usepackage{lineno}

\allowdisplaybreaks[4]
\numberwithin{equation}{section}

\newcommand{\nd}{{\ensuremath d}}
\newcommand{\M}{{\mathcal M}_{u,p}}

\newcommand{\B}{\ensuremath{B^s_{p,q}}}
\newcommand{\F}{\ensuremath{F^s_{p,q}}}
\newcommand{\A}{\ensuremath{A^s_{p,q}}}
\newcommand{\at}{{A}_{p,q}^{s,\tau}}
\newcommand{\bt}{{B}_{p,q}^{s,\tau}}
\newcommand{\ft}{{F}_{p,q}^{s,\tau}}
\newcommand{\btt}{{B}_{p,q}^{s,\varphi}}
\newcommand{\ftt}{{F}_{p,q}^{s,\varphi}}

\newcommand{\btts}{{b}_{p,q}^{s,\varphi}}
\newcommand{\ftts}{{f}_{p,q}^{s,\varphi}}

\newcommand{\Att}{{A}_{p,q}^{s,\varphi}}
\newcommand{\att}{{a}_{p,q}^{s,\varphi}}

\newcommand{\ajm}{a_{Q_{j,m}}}
\newcommand{\kjm}{\chi_{Q_{j,m}}}

\newcommand{\lql}{\ell^q(L_\varphi^p(\rd))}
\newcommand{\lpl}{L_\varphi^p(\ell^q(\rd))}

\newcommand{\lz}{\lambda}

\newcommand{\lzjm}{\lambda_{Q_{j,m}}}

\newcommand{\qjm}{Q_{j,m}}

\newcommand{\mq}{\mathcal{Q}}

\newcommand{\dint}{\mathrm{d}}
\newcommand{\Dd}{\;\mathrm{D}}
\newcommand{\supp}{\mathrm{supp}\,}

\newcommand{\eb}{\hookrightarrow}

\newcommand{\nat}{\ensuremath{\mathbb{N}}}
\newcommand{\no}{\ensuremath{\nat_0}}
\newcommand{\real}{\ensuremath{{\mathbb R}}}
\newcommand{\rd}{\ensuremath{{\mathbb R}^\nd}}
\newcommand{\zd}{\ensuremath{\mathbb{Z}^\nd}}
\newcommand{\sd}{{\mathcal S}(\mathbb{R}^\nd)}
\newcommand{\sdd}{{\mathcal S}'(\mathbb{R}^\nd)}
\newcommand{\Gp}{{\mathcal G}_p}

\newcommand{\bmo}{\ensuremath \mathrm{bmo}}

\newcommand{\mhl}{\mathcal{M}_{\mathrm{HL}}}
\newcommand{\jjp}{j_P\vee0}
\newcommand{\fuff}{\mathcal{F}^{-1}(\theta_j\mathcal{F}f)}

\newcommand{\whole}[1]{{\ensuremath\lfloor #1 \rfloor}}

\newtheorem{theorem}{Theorem}[section]
\newtheorem{lemma}[theorem]{Lemma}

\newtheorem{proposition}[theorem]{Proposition}
\theoremstyle{definition}
\newtheorem{remark}[theorem]{Remark}
\newtheorem{example}[theorem]{Example}
\newtheorem{definition}[theorem]{Definition}

\newcommand{\ignore}[1]{}
\title{\bf Generalized Besov-type and Triebel-Lizorkin-type spaces}
\author{Dorothee D. Haroske and Zhen Liu}
\date{}

\begin{document}
\maketitle

\begin{center}
\begin{minipage}{13cm}
{\small {\bf Abstract.}\quad Let $0<p<\infty$, $0<q\leq\infty$, and $s\in\real$. We introduce a new type of generalized Besov-type spaces $B_{p,q}^{s,\varphi}(\mathbb{R}^d)$ and generalized Triebel-Lizorkin-type spaces $F_{p,q}^{s,\varphi}(\mathbb{R}^d)$, where $\varphi$ belongs to the class $\mathcal{G}_p$, that is, $\varphi:(0,\,\infty)\rightarrow(0,\,\infty)$ is nondecreasing and $t^{-d/p}\varphi(t)$ is nonincreasing in $t>0$. We establish several properties, including some embedding properties, of these spaces. We also obtain the atomic decomposition of the spaces $B_{p,q}^{s,\varphi}(\mathbb{R}^d)$ and $F_{p,q}^{s,\varphi}(\mathbb{R}^d)$.}
\end{minipage}
\end{center}

\section{Introduction}
In 1938, Morrey \cite{mor38} introduced the concept of Morrey spaces $\mathcal{M}_{u,p}(\rd)$, $0<p\leq u<\infty$, to study the local behavior of solutions to second order elliptic partial differential equations. These spaces can also be considered as extension of the scale of Lebesgue spaces, since for $p=u$ one has $\mathcal{M}_{p,p}(\rd) = L^p(\rd)$. They are also part of a wider class of Morrey-Campanato spaces, cf. \cite{Pee}.
Motivated by the famous paper \cite{KY} of Kozono and Yamazaki who introduced Besov-Morrey spaces $\mathcal{N}^s_{u,p,q}(\rd)$ and used them to study Navier-Stokes equations, the study of smoothness spaces based on Morrey spaces (instead of $L^p$ spaces) was strongly enhanced.  In \cite{TX} Tang and Xu introduced the corresponding Triebel-Lizorkin-Morrey spaces, thanks to establishing the Morrey version of the  Fefferman-Stein vector-valued inequality. We refer for references, historic details and further results to the monographs \cite{ysy10}, the nice surveys \cite{s12,s13}, and also the more recent paper \cite{ht22}. As mentioned above, when $p=u$, one regains the (classical) Besov and Triebel-Lizorkin spaces, $\B(\rd)$ and $\F(\rd)$, respectively. For more information about Besov spaces and Triebel-Lizorkin spaces, we refer to Triebel's series of monographs \cite{t83,t92,t06,t20}.
Another famous approach to \textit{smoothness Morrey spaces} are the (inhomogeneous) Besov-type and Triebel-Lizorkin-type spaces $\bt(\rd)$, $\ft(\rd)$, $s\in\real$, $0<p,q\leq\infty$ (with $p<\infty$ in case of $\ft(\rd)$), $\tau\geq 0$, which were introduced and intensively studied in \cite{ysy10}.
Their homogeneous versions were previously investigated by El Baraka in \cite{ElBaraka1,ElBaraka2, ElBaraka3}, and also by Yuan and Yang \cite{yy08,yy10}.  Considering $\tau=0$, one recovers the classical Besov and Triebel-Lizorkin spaces again. These two scales of  \textit{smoothness Morrey spaces} are closely connected, as has been studied in many papers already, we refer to the above references for further details.

Now we return to the starting point, the Morrey spaces $\M(\rd)$, $0<p\leq u<\infty$. In the early 1990s, Mizuhara \cite{miz91} and Nakai \cite{nak94} introduced the generalized Morrey spaces $\mathcal{M}_{\varphi,p}(\rd)$, $0<p<\infty$, and $\varphi: (0, \infty)\rightarrow (0, \infty)$, independently.
Such spaces are sometimes more appropriate to describe the boundedness properties of operators than the classical Morrey spaces. A natural condition in the study of spaces $\mathcal{M}_{\varphi,p}(\rd)$ is the so-called $\Gp$ condition, that is, the function $\varphi$ belongs to the class $\mathcal{G}_p$, when $\varphi:(0,\infty)\rightarrow (0,\infty)$ is nondecreasing and $t^{-d/p}\varphi(t)$ is nonincreasing in $t>0$. Obviously the function $\varphi(t)=t^{d/u}$, $0<p\leq u<\infty$, satisfies this assumption, and in that case we recover the classical Morrey spaces.  We refer to \cite{sdh20a,sdh20b} for further information about the spaces and the historical remarks.
In \cite{nns16} Nakamura, Noi and Sawano introduced the generalized Besov-Morrey spaces  $\mathcal{N}_{\varphi,p,q}^s(\rd)$, cf. also \cite{AGNS}. They also proved the  atomic decomposition theorem for the spaces. The generalized Besov-Morrey spaces cover Besov-Morrey spaces and local Besov-Morrey spaces considered by Triebel \cite{t13,t14} as special cases. Quite recently, in \cite{hms22} also the wavelet decomposition and some embedding results for such spaces were obtained.

In this paper, we are introduce function spaces related to the generalized Morrey spaces in the spirit of the Besov-type and Triebel-Lizorkin-type spaces $\bt(\rd)$ and $\ft(\rd)$, but replacing this time the Morrey parameter $\tau\geq 0$ by an appropriate function $\varphi\in \Gp$. For that reason we shall call these spaces {\em generalized Besov-type spaces} $B_{p,q}^{s,\varphi}(\rd)$ and {\em generalized Triebel-Lizorkin-type spaces} $F_{p,q}^{s,\varphi}(\rd)$.

Note that in \cite{lsu13} Liang, Yang, Yuan, Sawano and Ullrich considered another type of generalized Besov-type spaces and Triebel-Lizorkin-type spaces via a general quasi-Banach space.
Moreover, Yang, Yuan and Zhuo \cite{yyz14} replaced the Lebesgue quasi-norm by the Musielak-Orlicz quasi-norm to generalize Besov-type spaces and Triebel-Lizorkin-type spaces. But this is different from our approach here. However, we stick to our notation of generalized Besov-type or Triebel-Lizorkin-type spaces, as later we shall also investigate their relation to Besov and Triebel-Lizorkin spaces built upon generalized  Morrey spaces which are called nowadays generalized Besov-Morrey spaces or generalized Triebel-Lizorkin-Morrey spaces, respectively. So this wording seems very  appropriate to us, despite the possible confusion with the spaces studied in \cite{lsu13}.

Let $\varphi\in\Gp$, $0<p<\infty$. Then we replace $|P|^\tau$ in the definition of the spaces $\bt(\rd)$ and $\ft(\rd)$,  cf. \cite[Definition 2.1]{ysy10} or Definition~\ref{def-Atau} below, by $\varphi(\ell(P))$, where $P$ is some dyadic cube with side length $\ell(P)$ and volume $|P|=\ell(P)^d$.  Then $\varphi(t)=t^{d\tau}$ belongs to $\Gp$ for all $\tau\in [0,\frac1p]$. Thus the new spaces $B_{p,q}^{s,\varphi}(\rd)$ coincide with $\bt(\rd)$ for all such parameters $\tau$, in particular, including $\tau=0$, that is, $\varphi\equiv 1$ where we arrive at $\B(\rd)$ again. The similar statement holds for spaces of Triebel-Lizorkin type.

We study basic properties of these new spaces, first embeddings results, and finally conclude our paper with the corresponding atomic decomposition theorem.
We always relate our results with the preceding results for special functions $\varphi$, and consider occasionally further examples of $\varphi\in\Gp$.

Here we essentially rely on findings and methods available for Besov-type and Triebel-Lizorkin-type spaces, as well as those developed for generalized Morrey spaces. Our approach is then to combine some ideas inspired by \cite{ysy10,nns16} with several techniques from classical theories of Besov spaces and Triebel-Lizorkin spaces in \cite{t83}.
Our results partly answer an open problem about generalizations posed in Sickel's survey \cite{s13}.

The paper is organized as follows. In Section~\ref{sec-prelim} we recall basic concepts and well-known definitions for later use.
Section~\ref{sec-fnt-s} concentrates on the underlying spaces $\lql$ and $\lpl$ which we later use for the definition of the spaces $B_{p,q}^{s,\varphi}(\rd)$ and $F_{p,q}^{s,\varphi}(\rd)$, respectively. We study fundamental properties like the boundedness of the Hardy-Littlewood maximal operator in these spaces.
In Section~\ref{sec-gen-type-sp} we define the generalized Besov-type spaces $B_{p,q}^{s,\varphi}(\rd)$ and generalized Triebel-Lizorkin-type spaces $F_{p,q}^{s,\varphi}(\rd)$ and establish several properties, especially embedding properties, of these function spaces. Finally, in Section~\ref{sec-atoms}, we obtain the atomic decomposition of $B_{p,q}^{s,\varphi}(\rd)$ and $F_{p,q}^{s,\varphi}(\rd)$, cf. Theorem~\ref{adbf}.

\section{Basic concepts}\label{sec-prelim}

First we fix some notation. Let $\nat$ be the collection of all natural numbers and $\no:=\nat\cup\{0\}$. Let $\rd$ be $d$-dimensional Euclidean space. By $\zd$ we denote the set of all lattice points in $\rd$ having integer components.
Let $\mathcal{S}(\rd)$ be the space of all Schwartz functions on $\rd$ endowed with the classical topology and denote by $\mathcal{S}'(\rd)$ its topological dual, namely, the space of all continuous linear functionals on $\mathcal{S}(\rd)$ endowed with the weak $\ast$-topology.
If $f\in\sd$, then $\mathcal{F}f(x):=(2\pi)^{-\frac{d}{2}}\int_{\rd}\mathrm{e}^{-ix\cdot\xi}f(\xi)\,\dint\xi$, $x\in\rd$, denotes the Fourier transform of $f$. The inverse Fourier transform of $f$ is given by $\mathcal{F}^{-1}f(\xi):=(2\pi)^{-\frac{d}{2}}\int_{\rd}\mathrm{e}^{ix\cdot\xi}{f}(x)\,\dint x$, $\xi\in\rd$.
For each cube $Q\subset\rd$, we denote \ignore{by $c(Q)$ the center of $Q$ and}  by $\ell(Q)$ the side length of $Q$ such that $|Q| = \ell(Q)^d$ is the  volume of the cube $Q$. For $x\in\rd$ and $r\in(0,\infty)$ we denote by $Q(x,r)$ the compact cube centred at $x$ with side length $r$, whose sides are parallel to the axes of coordinates.
Let $\mathcal{Q}$ denote the set of all dyadic cubes in $\rd$, namely, $\mathcal{Q}:=\{Q_{j,k}:=2^{-j}([0,1)^d+k):j\in\mathbb{Z}, k\in\zd\}$. For all $Q\in\mathcal{Q}$, let $j_Q:=-\log_2\ell(Q)$, and let $j_Q\vee 0:=\max(j_Q,\,0)$.
All unimportant positive constants will be denoted by $C$, occasionally with subscripts.
By the notation $A\lesssim B$, we mean that there exists a positive constant $c$ such that $A\leq c B$, whereas the symbol $A\sim B$ stands for $A\lesssim B\lesssim A$. Given two (quasi-)Banach spaces $X$ and $Y$, we write $X\eb Y$ if $X\subset Y$ and the natural embedding of $X$ into $Y$ is continuous.
Let $a\in\real$. We define $a_+:=\max(a,0)$ and $\whole{a}:=\max\{k\in\mathbb{Z}: k\leq a\}$. Finally, set $\sigma_p:=\nd(\frac{1}{p}-1)_+$ and $\sigma_{p,q}:=\max(\sigma_p,\sigma_q)$.

Let us now recall the definitions of Besov spaces, Triebel-Lizorkin spaces, Besov-type spaces and Triebel-Lizorkin-type spaces.
Let $\theta_0\in\sd$ with
\begin{align}\label{3e1}
   \theta_0(x)=1\quad\text{if}\quad |x|\leq 1\qquad\text{and}\qquad
  \theta_0(x)=0\quad\text{if}\quad |x|\geq 3/2,
\end{align}
and let
\begin{align}\label{3e2}
    \theta_k(x)=\theta_0(2^{-k}x)-\theta_0(2^{-k+1}x),\quad x\in\rd,\quad k\in\nat.
\end{align}
Since $\sum_{j=0}^{\infty}\theta_j(x)=1$, for $x\in\rd$, then $\theta:=\{\theta_j\}_{j\in\no}$ forms a \emph{smooth dyadic resolution of unity}.

Given a sequence of measurable functions $G=\{g_j\}_{j\in\no}$, we define
\begin{align*}
    \|G\mid \ell^q(L^p(\rd))\|:=\left(\sum_{j=0}^\infty\|g_j\mid L^p(\rd)\|^q\right)^\frac{1}{q},\quad
    \|G\mid L^p(\ell^q(\rd))\|:=\left\|\left(\sum_{j=0}^\infty|g_j|^q\right)^\frac{1}{q}\mid L^p(\rd)\right\|
\end{align*}
with the usual modifications for $q=\infty$.

\begin{definition}[\cite{t83}]
    Let $s\in\real$, $0<q\leq\infty$ and $\theta=\{\theta_j\}_{j\in\no}$ be the above dyadic resolution of unity.
    \begin{enumerate}[\bfseries\upshape  (i)]
        \item Let $0<p\leq\infty$. The Besov space $B_{p,q}^s(\rd)$ is defined to be the set of all $f\in\sdd$ such that
        \begin{align*}
          \|f\mid B_{p,q}^s(\rd)\|:=\|2^{j s}\fuff\mid \ell^q(L^p(\rd))\|<\infty
        \end{align*}
        \item Let $0<p<\infty$. The Triebel-Lizorkin space $F_{p,q}^s(\rd)$ is defined to be the set of all $f\in\sdd$ such that
        \begin{align*}
          \|f\mid B_{p,q}^s(\rd)\|:=\|2^{j s}\fuff\mid L^p(\ell^q(\rd))\|<\infty
        \end{align*}
with the usual modifications for $q=\infty$.
    \end{enumerate}
\end{definition}

\begin{remark}
The spaces $\B(\rd)$ and $\F(\rd)$ are independent of the choice of the dyadic resolution of unity $\theta=\{\theta_j\}_{j\in\no}$.
  Recall that the Besov spaces coincide in case of $p=q=\infty$, $s>0$, with the H\"older-Zygmund spaces $\mathcal{C}^s(\rd)$, in the sense of equivalent norms,
  \begin{equation}
    B^s_{\infty,\infty}(\rd) = \mathcal{C}^s(\rd), \quad s>0,
    \end{equation}
    cf. \cite[Thm. 2.5.7]{t83}. Obviously, $B^s_{p,p}(\rd)=F^s_{p,p}(\rd)$ for all admitted $s$ and $p$. Moreover, the Triebel-Lizorkin spaces cover the (fractional) Sobolev spaces for $1<p<\infty$, $s\in\real$, $k\in\no$, that is,
    \begin{equation}
      F^s_{p,2}(\rd)=H^s_p(\rd), \quad F^k_{p,2}(\rd) = W^k_p(\rd).
    \end{equation}
For more properties and characterizations of the above spaces we refer the reader to \cite{fj90} and the series of monographs \cite{t83,t92,t06} and references therein. Note that we usually assume $p<\infty$ in case of the spaces $F^s_{p,q}(\rd)$, though some of the properties remain valid in case of $p=\infty$, too, cf. \cite{t20}, but more care is needed.
\end{remark}

Similarly, assume that $0\leq\tau<\infty$, given a sequence of measurable functions $G=\{g_j\}_{j\in\no}$, we define
\begin{align*}
    \|G\mid \ell^q(L_\tau^p(\rd))\|:=\sup_{P\in\mq}\frac{1}{|P|^\tau}\left\{\sum_{j=\jjp}^\infty\left[\int_P
            |g_j(x)|^p\,\dint x\right]^\frac{q}{p}\right\}^\frac{1}{q},\\
    \|G\mid L_\tau^p(\ell^q(\rd))\|:=\sup_{P\in\mq}\frac{1}{|P|^\tau}\left\{\int_P\left[\sum_{j=\jjp}^\infty
            |g_j(x)|^q\right]^\frac{p}{q}\,\dint x\right\}^\frac{1}{p}.
\end{align*}

\begin{definition}[\cite{ysy10}]\label{def-Atau}
    Let $s\in\real$, $0<q\leq\infty$, $0\leq\tau<\infty$ and $\theta=\{\theta_j\}_{j\in\no}$ be the above dyadic resolution of unity.
    \begin{enumerate}[\bfseries\upshape  (i)]
        \item Let $0<p\leq\infty$. The Besov-type space $B_{p,q}^{s,\tau}(\rd)$ is defined by the set of all $f\in\sdd$ such that
        \begin{align*}
          \|f\mid B_{p,q}^{s,\tau}(\rd)\|:=\|2^{j s}\fuff\mid \ell^q(L_\tau^p(\rd))\|<\infty.
        \end{align*}
        \item Let $0<p<\infty$. The Triebel-Lizorkin-type space $F_{p,q}^{s,\tau}(\rd)$ is defined by the set of all $f\in\sdd$ such that
        \begin{align*}
          \|f\mid F_{p,q}^{s,\tau}(\rd)\|:=\|2^{j s}\fuff\mid L_\tau^p(\ell^q(\rd))\|<\infty.
        \end{align*}
    \end{enumerate}
\end{definition}

\noindent{\em Convention.}~We adopt the nowadays usual custom to write $\A(\rd)$ instead of $\B(\rd)$ or $\F(\rd)$, and $\at(\rd)$ instead  of $\bt(\rd)$ or $\ft(\rd)$, respectively, when both scales of spaces are meant simultaneously in some context.

\begin{remark}
  It is easily seen that $B_{p,q}^{s,0}(\rd)=B_{p,q}^s(\rd)$ and $F_{p,q}^{s,0}(\rd)=F_{p,q}^s(\rd)$.
It is also known from \cite[Proposition~2.6]{ysy10} that
\begin{equation} \label{010319}
A^{s,\tau}_{p,q}(\rd) \hookrightarrow B^{s+d(\tau-\frac1p)}_{\infty,\infty}(\rd).
\end{equation}
This result was even strengthened in \cite{yy13b} as follows. Assume that $s\in\real$, $\tau \geq 0$, $0<p,q\leq\infty$ (with $p<\infty$ in the $F$-case). If either $\tau>\frac1p$ or $\tau=\frac1p$ and $q=\infty$, then \begin{equation}
  \at(\rd) = B^{s+d(\tau-\frac1p)}_{\infty,\infty}(\rd).
\end{equation}
Moreover, in the limiting case, if $p,q\in(0,\infty)$ and $\tau=\frac1p$, then
	\begin{equation}\label{ftbt}
	F^{s,\, \frac{1}{p} }_{p\, ,\,q}(\rd) \, = \, F^{s}_{\infty,\,q}(\rd)\, = \, B^{s,\, \frac1q }_{q\, ,\,q}(\rd) \, ;
	\end{equation}
	cf. \cite[Propositions~3.4,~3.5]{s12} and \cite[Remark~10]{s13}.
For more properties and characterizations as well as further remarks about the history of the above spaces we refer the reader to \cite{ysy10,s12,s13} and the references therein.
      \end{remark}

\begin{remark}\label{bmo-def}
Recall that the space $\bmo(\rd)$ is covered by the above scale. More precisely, consider the local (non-homogeneous) space of functions of bounded mean oscillation, $\bmo(\rd)$, consisting of all locally integrable
functions $\ f $ satisfying that
\begin{equation*}
 \left\| f \right\|_{\bmo}:=
\sup_{|Q|\leq 1}\; \frac{1}{|Q|} \int\limits_Q |f(x)-f_Q| \dint x + \sup_{|Q|>
1}\; \frac{1}{|Q|} \int\limits_Q |f(x)| \dint x<\infty,
\end{equation*}
where $ Q $ appearing in the above definition runs over all cubes in $\rd$, and $ f_Q $ denotes the mean value of $ f $ with
respect to $ Q$, namely, $ f_Q := \frac{1}{|Q|} \;\int_Q f(x)\dint x$,
cf. \cite[2.2.2(viii)]{t83}. The space $\bmo(\rd)$ coincides with $F^{0}_{\infty, 2}(\rd)$,  cf. \cite[Theorem~2.5.8/2]{t83}. 
Hence the above result \eqref{ftbt} implies, in particular,
\begin{equation}\label{ft=bmo}
\bmo(\rd)= F^{0}_{\infty,2}(\rd)= F^{0, 1/p}_{p, 2}(\rd)= {B^{0, 1/2}_{2, 2}(\rd)}, \quad 0<p<\infty.
\end{equation}
\end{remark}

\begin{remark}\label{T-hybrid}
In contrast to this approach, Triebel followed the original Morrey-Campanato ideas to develop local spaces $\mathcal{L}^r\A(\rd)$ in \cite{t13}, and so-called `hybrid' spaces $L^r\A(\rd)$ in \cite{t14}, where $0<p<\infty$, $0<q\leq\infty$, $s\in\real$, and $-\frac{d}{p}\leq r<\infty$. This construction is based on wavelet decompositions and also combines local and global elements. However, Triebel proved in \cite[Theorem~3.38]{t14} that
$L^r\A(\rd) = \at(\rd)$, with $\tau=\frac1p+\frac{r}{d}$,
in all admitted cases.
\end{remark}

If one considers the generalized Morrey spaces $\mathcal{M}_{\varphi,p}(\rd)$, $0<p<\infty$, it is usually required that $\varphi\in\Gp$. Though we do not study the space $\mathcal{M}_{\varphi,p}(\rd)$ at the moment, our intention is to investigate in the future the relation between generalized Besov-Morrey spaces, say, and spaces of generalized Besov-type, similarly for the Triebel-Lizorkin spaces. Thus we concentrate on the class $\Gp$ from the very beginning and first recall its definition.

\begin{definition}[\cite{sdh20b}]
    Let $0<p<\infty$. Denote by $\Gp$ the set of all nondecreasing functions $\varphi:(0,\,\infty)\rightarrow(0,\,\infty)$ such that
\begin{align}\label{gpi}
t^{-\frac{d}{p}}\varphi(t)\geq s^{-\frac{d}{p}}\varphi(s),
\end{align}
for all $0<t\leq s<\infty$.
\end{definition}

\begin{remark}
  A justification for the use of the class $\Gp$ comes from \cite[Lemma 2.2]{nns16}. Let $0<p<\infty$ and $\varphi:(0,\infty)\rightarrow [0,\infty)$ be a function satisfying $\varphi(t_0)\neq 0$ for some $t_0>0$. Then $\mathcal{M}_{\varphi,p}(\rd)\neq \{0\}$ if, and only if, $\sup_{t>0} \varphi(t)  \min (t^{-\frac{\nd}{p}},1) < \infty$. Moreover, if $\displaystyle\sup_{t>0} \varphi(t)  \min (t^{-\frac{d}{p}},1) < \infty$, then there exists $\varphi^*\in\Gp$ such that $\mathcal{M}_{\varphi,p}(\rd)  = {\mathcal M}_{\varphi^*,p}(\rd) $ in the sense of equivalent (quasi-)norms.
  One can easily check that $\mathcal{G}_{p_1}\subset \mathcal{G}_{p_2}$ if $0<p_2\le p_1<\infty$. Note that we may often assume $\varphi(1)=1$ for convenience.
\end{remark}

We refer the reader to \cite[Section 12.1.2]{sdh20b} for more details about the class $\Gp$. Here we illustrate some examples from \cite{sdh20b} for later use.
Other examples   can be found e.g. in \cite[Ex.~3.15]{Saw18}.

\begin{example}
  \begin{enumerate}[\bfseries\upshape  (i)]
\item Let $u\in\real$, $0<p<\infty$ and let $\varphi(t)=t^u$ for $t>0$. Then $\varphi$ belongs to $\Gp$
if, and only if, $0\leq u\leq \frac{d}{p}$.
\item Let $0<u,\,v<\infty$. Then
$$
\varphi_{u,v}(t)=\begin{cases}
t^{\frac{d}{u}},\quad{\text{if}}\quad t\leq 1\\
t^{\frac{d}{v}},\quad{\text{if}}\quad t> 1
\end{cases}
$$
belongs to $\Gp$ with $p=\min(u,v)$. In particular, taking $u=v$, the function $\varphi(t)=t^{\frac{d}{u}}$
belongs to $\Gp$ whenever $0<p\leq u<\infty$.
\item Let $0<p<\infty$, $a\leq 0$ and $L\gg 1$, then $\varphi(t)=t^{\frac{d}{p}}(\log(L+t))^a$ for $t>0$ belongs to $\Gp$.
\item Let $0<p<\infty$, $0<u\ll 1$, and let $\varphi(t)=t^u(\log(e+t))^{-1}$ for $t>0$. Then $\varphi\notin\Gp$.
\end{enumerate}
\end{example}

\section{Some preparation: the spaces $\ell^q(L_\varphi^p(\rd))$ and $L_{\varphi}^p(\ell^q(\rd))$
}\label{sec-fnt-s}

Motivated by \cite{t83} and parallel to Definition~\ref{def-Atau}, we first investigate the following two underlying function spaces, which will be very helpful to study the spaces $B_{p,q}^{s,\varphi}(\rd)$ and $F_{p,q}^{s,\varphi}(\rd)$.

\begin{definition}\label{lpql}
  Let $0<p<\infty$, $0<q\leq\infty$ and $\varphi\in\Gp$.
  \begin{enumerate}[\bfseries\upshape  (i)]
    \item  The space $\ell^q(L_\varphi^p(\rd))$ is defined to be the set of all sequences $G:=\{g_j\}_{j\in\no}$ of measurable functions on $\rd$ such that
\begin{align*}
    \|G\mid\ell^q(L_\varphi^p(\rd))\|:=\sup_{P\in\mq}\frac{1}{\varphi(\ell(P))}\left\{\sum_{j=j_P\vee0}^\infty\left(\int_P|g_j(x)|^p\,\dint x\right)^{\frac{q}{p}}\right\}^{\frac{1}{q}}
\end{align*}
is finite (with the usual modification for $q=\infty$).
    \item Assume that there exists $\varepsilon>0$ such that
    \begin{align}\label{rintc}
        t^{\varepsilon-\frac{d}{p}}\varphi(t)\lesssim r^{\varepsilon-\frac{d}{p}}\varphi(r)
    \end{align}
    holds for $t\geq r$, when $0<q<\infty$. The space $L_{\varphi}^p(\ell^q(\rd))$ is defined to be the set of all sequences $G:=\{g_j\}_{j\in\no}$ of measurable functions on $\rd$ such that
\begin{align*}
    \|G\mid L_{\varphi}^p(\ell^q(\rd)\|:=\sup_{P\in\mq}\frac{1}{\varphi(\ell(P))}\left\{\int_P\left(\sum_{j=j_P\vee0}^\infty|g_j(x)|^q\right)^{\frac{p}{q}}\,\dint x\right\}^{\frac{1}{p}}
\end{align*}
is finite (with the usual modification for $q=\infty$).
\end{enumerate}
\end{definition}

\begin{remark}\label{rem1}
  \begin{enumerate}[\bfseries\upshape  (i)]
    \item If $\varphi(\ell(P)):=1$, then the spaces $\lpl$ and $\lql$ become, respectively, the spaces $L^p(\ell^q(\rd))$ and $\ell^q(L^p(\rd))$. They are quasi-Banach spaces.
    \item  If $\varphi(\ell(P)):=|P|^\tau$, then the spaces $\lql$ and $\lpl$ become, respectively, the spaces  $\ell^q(L^p_\tau(\rd))$ and $L^p_\tau(\ell^q(\rd))$, where $\tau\geq 0$.
    In this case, if $\tau>1/p$ and the summation would start always with $j=0$, then only the function $g=0$ a.e. belongs to such spaces due to Lebesgue's differentiation theorem (or Lebesgue points arguments).
    Thus, it is necessary to start the summation with respect to $j$ in dependence on the size of the dyadic cube $P$; see also \cite[Remark 2.1]{ysy10}.
\end{enumerate}
\end{remark}
\subsection{$\varepsilon$-Assumption}
The $\varepsilon$-assumption \eqref{rintc} is a crucial condition to obtain the vector-valued maximal inequality for the spaces $\lpl$;
see Theorem~\ref{mhlp} (ii) below.
We are interested in its characterization. The following two propositions are corollaries of \cite[Lemma 2]{nak94} and \cite[Proposition 2.7]{nns16}; see also \cite[Section 12.2.1]{sdh20b}.

\begin{proposition}\label{pp1}
    Let $C$ be a positive constant and $\varphi$ be a nonnegative locally integrable function which satisfies for all $r>0$ that
    \begin{align}\label{nintc}
        \int_{r}^\infty\frac{\varphi(t)}{t^{\frac{d}{p}+1}}\,\dint t\leq \frac{C\varphi(r)}{r^\frac{d}{p}}.
    \end{align}
    Then, for all 
    $\varepsilon>0$, 
    \begin{align}\label{nintc1}
        \int_{r}^\infty\frac{\varphi(t)}{t^{\frac{d}{p}+1-\varepsilon}}\,\dint t
        \leq\frac{C}{1-\varepsilon C}\cdot\frac{\varphi(r)}{r^{\frac{d}{p}-\varepsilon}}.
    \end{align}
\end{proposition}

\begin{proof}
    Let
    \begin{align*}
        \Phi(r):=\int_{r}^\infty\frac{\varphi(t)}{t^{\frac{d}{p}+1}}\,\dint t,\qquad 0<r<\infty.
    \end{align*}
    For $0<r<R$, by integration by parts and \eqref{nintc}, we have
    \begin{align*}
        \int_{r}^R\frac{\varphi(t)}{t^{\frac{d}{p}+1}}t^\varepsilon\,\dint t
        &=\left[-\Phi(t)t^\varepsilon\right]_r^R+\int_r^R\Phi(t)\varepsilon t^{\varepsilon-1}\,\dint t\\
        &= -\Phi(R)R^\varepsilon+\Phi(r)r^\varepsilon+\varepsilon\int_r^R\Phi(t)t^{\varepsilon-1}\,\dint t\\
        &\leq C\frac{\varphi(r)}{r^\frac{d}{p}}r^{\varepsilon}
        +\varepsilon C\int_r^R\frac{\varphi(t)}{t^{\frac{d}{p}+1}}t^\varepsilon\,\dint t.
    \end{align*}
    Therefore,
    \begin{align*}
        \int_{r}^R\frac{\varphi(t)}{t^{\frac{d}{p}+1-\varepsilon}}\,\dint t
        \leq\frac{C}{1-\varepsilon C}\cdot\frac{\varphi(r)}{r^{\frac{d}{p}-\varepsilon}}.
    \end{align*}
    Let $R\rightarrow\infty$, we obtain \eqref{nintc1}.
\end{proof}

\begin{proposition}\label{pp2}
    Let $\varphi$ be a nonnegative doubling measurable function such that
    $\varphi(s)\sim\varphi(r)$ for all $r, s>0$ with $\frac{1}{2}\leq\frac{r}{s}\leq2$.
    Then for some constants $C, \varepsilon>0$, conditions \eqref{rintc}, \eqref{nintc} and \eqref{nintc1} are equivalent.
    Moreover, if \eqref{rintc}-\eqref{nintc1} are satisfied, then for all $u>0$,
\begin{align}\label{nintc2}
    \int_{r}^\infty\frac{\varphi(t)^u}{t^{\frac{d}{p}u+1}}\,\dint t
    \leq \frac{C_0\varphi(r)^u}{r^{\frac{d}{p}u}},\qquad 0<r<\infty,
\end{align}
where $C_0$ is independent of $r$ and depends only on $u$ and implicit constant in \eqref{rintc}.
\end{proposition}

\begin{proof}
    Assume \eqref{rintc}, then
    \begin{align*}
        \int_r^\infty\frac{\varphi(t)}{t^{\frac{d}{p}+1}}\,\dint t
        =\int_r^\infty\frac{\varphi(t)}{t^{\frac{d}{p}-\varepsilon}}\frac{\dint t}{t^{1+\varepsilon}}
        \lesssim\int_r^\infty\frac{\varphi(r)}{r^{\frac{d}{p}-\varepsilon}}\frac{\dint t}{t^{1+\varepsilon}}
        \sim\frac{\varphi(r)}{r^\frac{d}{p}},
    \end{align*}
    which implies \eqref{nintc}. By Proposition~\ref{pp1}, \eqref{nintc} implies \eqref{nintc1}.
    Finally, we prove \eqref{nintc1} implies \eqref{rintc}. By \eqref{nintc1} and the doubling property of $\varphi$,
    \begin{align*}
        t^{\varepsilon-\frac{d}{p}}\varphi(t)\lesssim\int_t^{2t}y^{\varepsilon-\frac{d}{p}-1}\varphi(y)\,\dint y
        \lesssim r^{\varepsilon-\frac{d}{p}}\varphi(r).
    \end{align*}
    Since $\varphi(\cdot)^u(\cdot)^{-\frac{d}{p}u}$ satisfies \eqref{rintc}, then \eqref{rintc} implies \eqref{nintc2} as well.
\end{proof}

\subsection{Some Properties}
We investigate the boundedness property of the maximal operator given by
\begin{align}\label{maxi}
 \mhl(g)(x)=\sup_{Q\in\mq_x(\rd)}\frac{1}{|Q|}\int_Q|g(z)|\,\dint z,
\end{align}
where $\mq_x(\rd)$ stands for the collection of all cubes that contain $x\in\rd$.

We need following lemma to obtain the main results.

\begin{lemma}\label{usel1}
For all $f\in L^0(\rd)$ and cubes $Q$, we have
\begin{align*}
    \mhl[\chi_{\rd\backslash5Q}f](x)\lesssim\sup_{R\in\mq^\sharp(Q)}\frac{1}{|R|}\int_R|f(z)|\,\dint z
    \lesssim\sum_{k=0}^\infty\frac{1}{|2^{k}Q|}\int_{2^{k}Q}|f(z)|\,\dint z
    \qquad(x\in Q),
\end{align*}
where $\mq^\sharp(Q)$ stands for the collection of all cubes containing $Q$.
\end{lemma}

\begin{proof}
The proof is similar to the proof of \cite[Lemma 4.2]{shg15}; see also \cite[Lemma 130]{sdh20a}. For the reader's convenience, we give some details.

For all cubes $Q\ni x$, it is obvious that $\mq_x(\rd)\subseteq\mq^\sharp(Q)$. By \eqref{maxi},
   \begin{align*}
        \mhl[\chi_{\rd\backslash5Q}f](x)=\sup_{Q^\ast\in\mq_x(\rd)}\frac{1}{|Q^\ast|}
        \int_{[\cup_{k=0}^\infty(10\cdot 2^kQ)\backslash(5\cdot 2^kQ)]\cap Q^\ast}|f(z)|\,\dint z.
    \end{align*}
    In order that $x\in Q^\ast\cap Q$ and $Q^\ast\cap(\rd\backslash5Q)\neq\varnothing$, we have $\ell(Q^\ast)\geq 2\ell(Q)$. Thus,
    \begin{align*}
        \mhl[\chi_{\rd\backslash5Q}f](x)=\sup_{Q^\ast\in\mq_x(\rd);\, \ell(Q^\ast)\geq 2\ell(Q)}
        \frac{1}{|Q^\ast|}\int_{[\cup_{k=0}^\infty(10\cdot 2^kQ)\backslash(5\cdot 2^kQ)]\cap Q^\ast}|f(z)|\,\dint z.
    \end{align*}
    For any $Q^\ast\in\mq_x(\rd)$, choose a cube $R$ such that $\ell(R)=2\ell(Q^\ast)$ and $R$ and $Q^\ast$ have the same center.
    Then, $Q^\ast\subset R$, $Q\subset R$ and $R\in\mq_x(\rd)$. Thus,
    \begin{align*}
        \mhl[\chi_{\rd\backslash5Q}f](x)&\leq\sup_{R\in\mq^\sharp(Q)}\frac{2^d}{|R|}
        \int_{[\cup_{k=0}^\infty(10\cdot 2^kQ)\backslash(5\cdot 2^kQ)]\cap R}|f(z)|\,\dint z\\
        &\leq\sum_{k=0}^\infty\sup_{R\in\mq^\sharp(Q)}\frac{2^d}{|R|}
        \int_{(10\cdot 2^kQ \backslash 5\cdot 2^kQ)\cap R}|f(z)|\,\dint z\\
        &\leq\sum_{k=0}^\infty\frac{4^d}{|10\cdot 2^kQ|}\int_{10\cdot 2^kQ}|f(z)|\,\dint z,
    \end{align*}
    which completes the proof of Lemma~\ref{usel1}.
\end{proof}

\begin{theorem}\label{mhlp}
Let $1<p<\infty$, $1<q\leq\infty$ and $\varphi\in\Gp$.
  \begin{enumerate}[\bfseries\upshape  (i)]
    \item  Then for any sequence $\{f_j\}_{j\in\no}$ of $\lql$-functions,
    \begin{align}\label{maxp1}
        \|\{\mhl f_j\}_{j\in\no}\mid\lql\|\lesssim\|\{f_j\}_{j\in\no}\mid\lql\|.
    \end{align}
    \item  Assume in addition that $\varphi$ satisfies \eqref{rintc} when $0<q<\infty$.
    Then for any sequence $\{f_j\}_{j\in\no}$ of $\lpl$-functions,
    \begin{align}\label{maxp2}
        \|\{\mhl f_j\}_{j\in\no}\mid\lpl\|\lesssim\|\{f_j\}_{j\in\no}\mid\lpl\|.
    \end{align}
\end{enumerate}
\end{theorem}

\begin{proof}
We first prove (i). Let $P\in\mathcal{Q}$ be a fixed cube. Then it suffices to show
\begin{align*}
    \frac{1}{\varphi(\ell(P))}\left\{\sum_{j=\jjp}^\infty
    \left(\int_P|\mhl f_j(x)|^p\,\dint x\right)^{\frac{q}{p}}
    \right\}^{\frac{1}{q}}\lesssim\|\{f_j\}_{j\in\no}\mid\lql\|.
\end{align*}
Write $f_{j,1}:=\chi_{5P}f_j$ and $f_{j,2}:=f_j-f_{j,1}$. The estimate of $f_{j,1}$ is simple. By $\mhl$ is $L^p(\rd)$-bounded and \eqref{gpi},
\begin{align*}
    \frac{1}{\varphi(\ell(P))}\left\{\sum_{j=\jjp}^\infty\left(\int_P|\mhl f_{j,1}(x)|^p\,\dint x\right)^{\frac{q}{p}}\right\}^{\frac{1}{q}}
    &\leq\frac{1}{\varphi(\ell(P))}\left\{\sum_{j=\jjp}^\infty\left(\int_{\rd}|\mhl f_{j,1}(x)|^p\,\dint x\right)^{\frac{q}{p}}\right\}^{\frac{1}{q}}\\
    &\lesssim\frac{1}{\varphi(\ell(5P))}\left\{\sum_{j=j_{5P}\vee0}^\infty\left(\int_{5P} |f_j(x)|^p\,\dint x\right)^{\frac{q}{p}}\right\}^{\frac{1}{q}}\\
    &\lesssim\|\{f_j\}_{j\in\nat}\mid\lql\|.
\end{align*}
To estimate $f_{j,2}$, we use Lemma~\ref{usel1}, \eqref{gpi} and H\"older's inequality. Then
\begin{align*}
    \frac{1}{\varphi(\ell(P))}\left\{\sum_{j=\jjp}^\infty\left(\int_P|\mhl f_{j,2}(x)|^p\,\dint x\right)^{\frac{q}{p}}\right\}^{\frac{1}{q}}
    &\lesssim\frac{|P|^{\frac{1}{p}}}{\varphi(\ell(P))}\left\{\sum_{j=\jjp}^\infty\left(\sup_{R\in\mq^\sharp(P)}\frac{1}{|R|}\int_R|f_j(z)|\,\dint z\right)^q\right\}^{\frac{1}{q}}\\
    &\lesssim\sup_{R\in\mq^\sharp(P)}\frac{1}{\varphi(\ell(R))}\left\{\sum_{j=j_R\vee0}^\infty\left(\int_R|f_j(z)|^p\,\dint z\right)^{\frac{q}{p}}\right\}^{\frac{1}{q}}\\
    &\lesssim\|\{f_j\}_{j\in\nat}\mid\lql\|.
\end{align*}
Hence we obtain \eqref{maxp1}.

Next we prove (ii). Write again $f_{j,1}:=\chi_{5P}f_j$ and $f_{j,2}:=f_j-f_{j,1}$. From an argument similar to the estimate of $f_{j,1}$ in (i) above and Fefferman-Stein vector-valued inequality \cite{fs71}, we deduce that
\begin{align*}
    \frac{1}{\varphi(\ell(P))}\left\{\int_P\left(\sum_{j=\jjp}^\infty|\mhl f_{j,1}(x)|^q\right)
    ^{\frac{p}{q}}\,\dint x\right\}^{\frac{1}{p}}\lesssim\|\{f_j\}_{j\in\no}\mid\lpl\|.
\end{align*}
To deal with $f_{j,2}$, by Lemma~\ref{usel1}, Minkowski's inequality and H\"older's inequality,
\begin{align}\label{pri}
    \frac{1}{\varphi(\ell(P))}\left\{\int_P\left(\sum_{j=\jjp}^\infty|\mhl f_{j,2}(x)|^q\right)
    ^{\frac{p}{q}}\,\dint x\right\}^{\frac{1}{p}}
    &\lesssim\sum_{k=0}^\infty\frac{|P|^\frac{1}{p}}{|2^kP|}\frac{1}{\varphi(\ell(P))}\int_{2^kP}\left(\sum_{j=j_{2^kP}\vee0}^\infty|f_j(z)|^q\right)^{\frac{1}{q}}\,\dint z\notag \\
    &\lesssim\sum_{k=0}^\infty\frac{|P|^\frac{1}{p}}{|2^kP|^\frac{1}{p}}\frac{\varphi(\ell(2^kP))}{\varphi(\ell(P))}\|\{f_j\}_{j\in\nat}\mid\lpl\|.
\end{align}
From \eqref{rintc}, we see that
\begin{align*}
    \sum_{k=0}^\infty\frac{|P|^\frac{1}{p}}{|2^kP|^\frac{1}{p}}\frac{\varphi(\ell(2^kP))}{\varphi(\ell(P))}
    \lesssim\sum_{k=0}^\infty 2^{-k\varepsilon}\lesssim1.
\end{align*}
Inserting this in \eqref{pri}, we obtain \eqref{maxp2}. Hence we finish the proof of Theorem~\ref{mhlp}.
\end{proof}

\begin{theorem}\label{assth1}
Let $0<p<\infty$, $0<q\leq\infty$, $\varphi\in\Gp$ and $0<r<\min(p,q)$.
Let $\Omega:=\{\Omega_j\}_{j\in\no}$ be a sequence of compact subsets of $\rd$ and $D_j>0$ be the diameter of $\Omega_j$.
  \begin{enumerate}[\bfseries\upshape  (i)]
    \item Then for all $\{f_j\}_{j\in\no}\in\lql$,
    \begin{align}\label{maxs1}
        \left\|\left\{\sup_{z\in\rd}\frac{|f_j(\cdot-z)|}{1+|D_jz|^{\frac{d}{r}}}\right\}_{j\in\no}\mid\lql\right\|
        \lesssim\|\{f_j\}_{j\in\no}\mid\lql\|.
    \end{align}
    \item  Assume in addition that $\varphi$ satisfies \eqref{rintc} when $0<q<\infty$. Then for all $\{f_j\}_{j\in\no}\in\lpl$,
     \begin{align}\label{maxs2}
        \left\|\left\{\sup_{z\in\rd}\frac{|f_j(\cdot-z)|}{1+|D_jz|^{\frac{d}{r}}}\right\}_{j\in\no}\mid\lpl\right\|
        \lesssim\|\{f_j\}_{j\in\no}\mid\lpl\|.
    \end{align}
\end{enumerate}
\end{theorem}

\begin{proof}
By similarity, we only prove (i). Let $\{f_j\}_{j\in\no}\in\lql$, $y^j\in\Omega_j$ and $h_j(x)=e^{-ixy^j}f_j(x)$,
then we have $(\mathcal{F}h_j)(x)=(\mathcal{F}f_j)(x+y^j)$.
Therefore, $\supp\mathcal{F}h_j\subset\Omega_j-y^j$.
If \eqref{maxs1} holds for $\{f_j\}_{j\in\no}\in\lql$, then it also does for $\{h_j\}_{j\in\no}$,
where $\Omega$ is replaced by $\{\Omega_j-y^j\}_{j\in\no}$, and the converse also holds.
Thus, we may assume that $0\in\Omega_j$. Then it is sufficient to prove \eqref{maxs1} with $\Omega_j=\{y:|y|\leq D_j\}$ and $D_j>0$.

If $\{f_j\}_{j\in\no}\in\lql$, then $f_j\in L^{p,\Omega_j}$,
where $L^{p,\Omega_j}:=\{f: f\in\sdd, \supp\mathcal{F}f\subset \Omega_j, \|f\mid L^p(\rd)\|<\infty\}$;
see \cite[p.22]{t83} for details.
If $\phi_j(x)=f_j(D_j^{-1}x)$, then $(\mathcal{F}\phi_j)(x)=D_j^d(\mathcal{F}f_j)(D_jx)$
and $\supp\mathcal{F}\phi_j\subset\{y:|y|\leq1\}$.
Hence, by the arguments in the first step of the proof of Theorem 1.4.1 in \cite{t83}, we obtain for all $x,z\in\rd$,
\begin{align}\label{eq0}
\frac{|\phi_j(x-z)|}{1+|z|^{\frac{d}{r}}}\lesssim[\mhl(|\phi_j|^r)(x)]^{\frac{1}{r}},
\end{align}
where the implicit constant in \eqref{eq0} is independent of $x, z$ and $j\in\no$. From \eqref{eq0}, we conclude that
\begin{align}\label{asse1}
    \frac{|f_j(x-z)|}{1+|D_jz|^{\frac{d}{r}}}\lesssim[\mhl(|f_j|^r)(x)]^{\frac{1}{r}}.
\end{align}
Assume first $q<\infty$. Note that $\frac{p}{r}>1$, $\frac{q}{r}>1$. Then, by \eqref{asse1} and \eqref{maxp1}, we have
\begin{align*}
    \left\|\left\{\sup_{z\in\rd}\frac{|f_j(\cdot-z)|}{1+|D_jz|^{\frac{d}{r}}}\right\}_{j\in\no}\mid\lql\right\|
    &\lesssim\|\left\{\mhl(|f_j|^r)\right\}_{j\in\no}\mid\ell^{\frac{q}{r}}(L_\varphi^{\frac{p}{r}}(\rd))\|^{\frac{1}{r}}\\
    &\lesssim\|\{f_j\}_{j\in\no}\mid\lql\|.
\end{align*}
When $q=\infty$, then $\frac{p}{r}>1$, \eqref{asse1} and \eqref{maxp1} yield
\begin{align*}
    \left\|\left\{\sup_{z\in\rd}\frac{|f_j(\cdot-z)|}{1+|D_jz|^
    {\frac{d}{r}}}\right\}_{j\in\no}\mid\ell^\infty(L_\varphi^p(\rd))\right\|
    &\lesssim\sup_{P\in\mq}\frac{1}{\varphi(\ell(P))}\sup_{j}\left\{
    \int_P|\mhl(|f_j|)^r(x)|^\frac{p}{r}\,\dint x\right\}^{\frac{r}{p}\cdot\frac{1}{r}}\\
    &\lesssim\|\{f_j\}_{j\in\no}\mid\ell^\infty(L_\varphi^p(\rd))\|.
\end{align*}
Hence we complete the proof.
\end{proof}

With the aid of Theorem~\ref{assth1}, we have the following multiplier theorem. Recall that if $s\in\real$, then
\begin{align*}
    H_2^s(\rd):=\{f:f\in\sdd,\,\|f\mid H_2^s(\rd)\|=\|(1+|x|^2)^{\frac{s}{2}}(\mathcal{F}f)(x)\mid L^2(\rd)\|<\infty\}.
\end{align*}

\begin{theorem}\label{thme}
Let $0<p<\infty$, $0<q\leq\infty$, $\varphi\in\Gp$ and $\varkappa>\frac{d}{2}+\frac{d}{\min(p,q)}$.
Let $\Omega:=\{\Omega_j\}_{j\in\no}$ be a sequence of compact subsets of $\rd$ and $D_j>0$ be the diameter of $\Omega_j$.
  \begin{enumerate}[\bfseries\upshape  (i)]
    \item  Then for all $\{f_j\}_{j\in\no}\in\lql$ and all sequences $\{\mu_j(x)\}_{j\in\no}\subset H_2^\varkappa(\rd)$,
    \begin{align*}
        \|\{\mathcal{F}^{-1}(\mu_j\mathcal{F}f_j)\}_{j\in\no}\mid\lql\|
        \lesssim\sup_j\|\mu_j(D_j\cdot)\mid H_2^\varkappa(\rd)\|\cdot\|\{f_j\}_{j\in\no}\mid\lql\|.
    \end{align*}
    \item  Assume in addition that $\varphi$ satisfies \eqref{rintc} when $0<q<\infty$.
    Then for all $\{f_j\}_{j\in\no}\in\lpl$ and all sequences $\{\mu_j(x)\}_{j\in\no}\subset H_2^\varkappa(\rd)$,
     \begin{align*}
        \|\{\mathcal{F}^{-1}(\mu_j\mathcal{F}f_j)\}_{j\in\no}\mid\lpl\|
        \lesssim\sup_j\|\mu_j(D_j\cdot)\mid H_2^\varkappa(\rd)\|\cdot\|\{f_j\}_{j\in\no}\mid\lpl\|.
    \end{align*}
\end{enumerate}
\end{theorem}

\begin{proof}
By similarity, we only prove (i). We claim that if $0<r<\min(p,q)$ and $\varkappa>\frac{d}{2}+\frac{d}{r}$, then
\begin{align}\label{usse1}
    \sup_{z\in\rd}\frac{|\mathcal{F}^{-1}(\mu_j\mathcal{F}f_j)(x-z)|}{1+|D_jz|^\frac{d}{r}}
    \lesssim\sup_{z\in\rd}\frac{|f_j(x-z)|}{1+|D_jz|^\frac{d}{r}}\|\mu_j(D_j\cdot)\mid H_2^\varkappa(\rd)\|.
\end{align}
In fact, by Remark 1.5.1.1 in \cite{t83}, we have
the inequality $1+|D_j(x-y)|^\frac{d}{r}\lesssim(1+|D_j(x-y-z)|^\frac{d}{r})(1+|D_jz|^\frac{d}{r})$
and after appropriate substitution of variables we obtain
\begin{align*}
    |\mathcal{F}^{-1}(\mu_j\mathcal{F}f_j)(x-z)|&\lesssim\int_{\rd}|(\mathcal{F}^{-1}\mu_j)(x-z-y)||f_j(y)|\,\dint y\\
    &\lesssim\sup_{t\in\rd}\frac{|f_j(t)|}{1+|D_j(x-t)|^\frac{d}{r}}\int_{\rd}|(\mathcal{F}^{-1}\mu_j)(x-z-y)|(1+|D_j(x-y)|^\frac{d}{r})\,\dint y\\
    &\lesssim\sup_{z\in\rd}|f_j(x-z)|\int_{\rd}|(\mathcal{F}^{-1}\mu_j)(y)|(1+|D_jy|^\frac{d}{r})\,\dint y.
\end{align*}
Dividing both sides by $1+|D_jz|^\frac{d}{r}$, we get
\begin{align}\label{usee2}
    \sup_{z\in\rd}\frac{|\mathcal{F}^{-1}(\mu_j\mathcal{F}f_j)(x-z)|}{1+|D_jz|^\frac{d}{r}}
    \lesssim\sup_{z\in\rd}\frac{|f_j(x-z)|}{1+|D_jz|^\frac{d}{r}}
    \int_{\rd}|(\mathcal{F}^{-1}\mu_j)(y)|(1+|D_jy|^\frac{d}{r})\,\dint y.
\end{align}
Using $(\mathcal{F}^{-1}\mu_j(D_j\cdot))(y)=D_j^{-d}(\mathcal{F}^{-1}\mu_j)(D_j^{-1}y)$
and $\varkappa>\frac{d}{2}+\frac{d}{r}$, we conclude
\begin{align*}
    \int_{\rd}|(\mathcal{F}^{-1}\mu_j)(y)|(1+|D_jy|^\frac{d}{r})\,\dint y
    &=\int_{\rd}|(\mathcal{F}^{-1}\mu_j(D_j\cdot))(y)|(1+|y|^\frac{d}{r})\,\dint y\\
    &\lesssim\left(\int_{\rd}(1+|y|^\varkappa)^2|(\mathcal{F}^{-1}\mu_j(D_j\cdot))(y)|^2\,\dint y\right)^\frac{1}{2}\\
    &\lesssim\|\{\mu_j(D_j\cdot)\}_{j\in\no}\mid H_2^\varkappa(\rd)\|.
\end{align*}
Inserting this estimate in \eqref{usee2}, we get \eqref{usse1}. Then by \eqref{usse1}, \eqref{maxs1} and the fact that $|\mathcal{F}^{-1}(\mu_j\mathcal{F}f_j)(x)|\leq\sup_{z\in\rd}\frac{|\mathcal{F}^{-1}(\mu_j\mathcal{F}f_j)(x-z)|}{1+|D_jz|^\frac{d}{r}}$, we obtain the desired result.
\end{proof}

\begin{remark}
In Theorem~\ref{thme}, $\mathcal{F}^{-1}(\mu_j\mathcal{F}f_j)$ is well defined; see Remark 1.5.1.1 and the proof of Theorem 1.6.3 in \cite{t83} for more details.
\end{remark}

We also need the following useful property.

\begin{proposition}\label{ggl}
    Let $0<p<\infty$, $0<q\leq\infty$, $\varphi\in\Gp$, $0<\gamma<\infty$ and $\{g_k\}_{k\in\no}$ be a sequence of measurable functions on $\rd$. For all $j\in\no$ and $x\in\rd$, let
    \begin{align*}
        G_j(x):=\sum_{k=0}^\infty2^{-|k-j|\gamma}g_k(x).
    \end{align*}
  \begin{enumerate}[\bfseries\upshape  (i)]
        \item  Then
        \begin{align*}
            \|\{G_j\}_{j\in\no}\mid\lql\|\lesssim\|\{g_k\}_{k\in\no}\mid\lql\|;
        \end{align*}
        \item  Assume in addition that $\varphi$ satisfies \eqref{rintc} when $0<q<\infty$. Then
        \begin{align*}
            \|\{G_j\}_{j\in\no}\mid\lpl\|\lesssim\|\{g_k\}_{k\in\no}\mid\lpl\|.
        \end{align*}
    \end{enumerate}
\end{proposition}

\begin{proof}
    By similarity, we only prove (ii). Let $P\in\mq$ and
    \begin{align*}
        {\rm{I}}_P:=\frac{1}{\varphi(\ell(P))}\left\{\int_P\left[\sum_{j=\jjp}^\infty
        \left(\sum_{k=0}^\infty2^{-|k-j|\gamma}|g_k(x)|\right)^q
        \right]^\frac{p}{q}\,\dint x\right\}^\frac{1}{p}.
    \end{align*}
Now we need to prove that
\begin{align*}
    {\rm{I}}_P&\lesssim\|\{g_k\}_{k\in\no}\mid\lpl\|,
\end{align*}
where the implicit constant in $\lesssim$ is independent of $P$ and $\{g_k\}_{k\in\no}$. By the triangle inequality,
\begin{align*}
    {\rm{I}}_P&\leq\frac{1}{\varphi(\ell(P))}\left\{\int_P\left[\sum_{j=\jjp}^\infty\left(\sum_{k=0}^j 2^{-(j-k)\gamma}|g_k(x)|\right)^q\right]^\frac{p}{q}\,\dint x\right\}^\frac{1}{p}\\
    &\quad+\frac{1}{\varphi(\ell(P))}\left\{\int_P\left[\sum_{j=\jjp}^\infty\left(\sum_{k=j+1}^\infty2^{-(k-j)\gamma}|g_k(x)|\right)^q\right]^\frac{p}{q}\,\dint x\right\}^\frac{1}{p}\\
    &=:{\rm{I}}_{P,1}+{\rm{I}}_{P,2}.
\end{align*}

\emph{Step 1.} If $0<q\leq1$, applying the monotonicity of the $\ell^q$-norm in $q$, we have
\begin{align*}
   {\rm{I}}_{P,1}&\leq\frac{1}{\varphi(\ell(P))}\left\{\int_P\left(\sum_{j=\jjp}^\infty\sum_{k=0}^j2^{-(j-k)\gamma q}|g_k(x)|^q\right)^\frac{p}{q}\,\dint x\right\}^\frac{1}{p}\\
    &=\frac{1}{\varphi(\ell(P))}\left\{\int_P\left(\sum_{k=0}^\infty\sum_{j=j_P\vee k}^\infty 2^{-(j-k)\gamma q}|g_k(x)|^q\right)^\frac{p}{q}\,\dint x\right\}^\frac{1}{p}.
\end{align*}
Note that
\begin{align*}
    \sum_{j=j_P\vee k}^\infty 2^{-(j-k)\gamma q}
    \leq \sum_{j=k}^\infty 2^{-(j-k)\gamma q}
    =\sum_{\ell=0}^\infty 2^{-\ell \gamma q}<\infty,
\end{align*}
then
\begin{align*}
    {\rm{I}}_{P,1}\lesssim \frac{1}{\varphi(\ell(P))}\left\{\int_P\left(\sum_{k=0}^\infty|g_k(x)|^q\right)^\frac{p}{q}\,\dint x\right\}^\frac{1}{p}\lesssim\|\{g_k\}_{k\in\no}\mid\lpl\|.
\end{align*}
For ${\rm{I}}_{P,2}$, we have
\begin{align*}
    {\rm{I}}_{P,2}&\leq\frac{1}{\varphi(\ell(P))}\left\{\int_P\left(\sum_{k=\jjp}^\infty\sum_{j=j_P\vee 0}^k 2^{-(k-j)\gamma q}|g_k(x)|^q\right)^\frac{p}{q}\,\dint x\right\}^\frac{1}{p}\\
    &\lesssim\frac{1}{\varphi(\ell(P))}\left\{\int_P\left(\sum_{k=\jjp}^\infty|g_k(x)|^q\right)^\frac{p}{q}\,\dint x\right\}^\frac{1}{p}\\
    &\lesssim\|\{g_k\}_{k\in\no}\mid\lpl\|.
\end{align*}

\emph{Step 2.} If $1<q\leq\infty$, choosing $\varepsilon\in(0,\gamma)$ and applying H\"older's inequality, we have
\begin{align*}
    {\rm{I}}_{P,1}&\leq\frac{1}{\varphi(\ell(P))}\left\{\int_P\left[
    \sum_{j=\jjp}^\infty\left(\sum_{k=0}^j 2^{-(j-k)(\gamma-\varepsilon)q}|g_k(x)|^q\right)
    \left(\sum_{k=0}^j 2^{-(j-k)\varepsilon q'}|g_k(x)|^{q'}\right)
    ^\frac{q}{q'}\right]^\frac{p}{q}\,\dint x\right\}^\frac{1}{p}\\
    &\lesssim \frac{1}{\varphi(\ell(P))}\left\{\int_P\left(\sum_{j=\jjp}^\infty
    \sum_{k=0}^{j}2^{-(j-k)(\gamma-\varepsilon)q}|g_k(x)|^q\right)^\frac{p}{q}\,
    \dint x\right\}^\frac{1}{p}
\end{align*}
and
\begin{align*}
    {\rm{I}}_{P,2}\lesssim \frac{1}{\varphi(\ell(P))}\left\{\int_P\left(\sum_{j=\jjp}^\infty
    \sum_{k=j+1}^{\infty}2^{-(j-k)(\gamma-\varepsilon)q}|g_k(x)|^q\right)^\frac{p}{q}\,\dint x\right\}^\frac{1}{p},
\end{align*}
where $q'$ denotes the conjugate index of $q$, namely, $\frac{1}{q'}+\frac{1}{q}=1$.

By similar arguments as in Step 1, we also have
\begin{align*}
    {\rm{I}}_P \leq{\rm{I}}_{P,1}+{\rm{I}}_{P,2}\lesssim\|\{g_k\}_{k\in\no}\mid\lpl\|.
\end{align*}
Hence we finish the proof of Proposition~\ref{ggl}.
\end{proof}

\begin{remark}
The case $\varphi:=1$ has been obtained in \cite[Lemma 2]{ryc99}.
Furthermore, the case $\varphi(\ell(P)):=|P|^\tau$ with $0\leq\tau<\infty$ is covered by \cite[Lemma 2.3]{yy10b}.
\end{remark}

\section{Generalized Besov-type and Triebel-Lizorkin-type spaces}\label{sec-gen-type-sp}

We now generalize Besov-type spaces and Triebel-Lizorkin-type spaces via a function $\varphi\in\Gp$.

\begin{definition}\label{d3}
Let $s\in\real$, $0<p<\infty$, $0<q\leq\infty$, $\varphi\in\Gp$ and $\theta=\{\theta_j\}_{j\in\no}$ be the above dyadic resolution of unity.
  \begin{enumerate}[\bfseries\upshape  (i)]
\item The generalized Besov-type space $\btt(\rd)$ is defined to be the set of all $f\in\sdd$ such that
    \begin{align}\label{bsn}
    \|f\mid B_{p,q}^{s,\varphi}(\rd)\|:=\|\{2^{js}\mathcal{F}^{-1}(\theta_j\mathcal{F}f)\}_{j\in\no}\mid\lql\|
    \end{align}
  is finite (with the usual modification for $q=\infty$).
\item Assume in addition that $\varphi$ satisfies \eqref{rintc} when $0<q<\infty$. The generalized Triebel-Lizorkin-type space $\ftt(\rd)$ is defined to be the set of all $f\in\sdd$ such that
    \begin{align}\label{fsn}
    \|f\mid F_{p,q}^{s,\varphi}(\rd)\|:=\|\{2^{js}\mathcal{F}^{-1}(\theta_j\mathcal{F}f)\}_{j\in\no}\mid\lpl\|
    \end{align}
   is finite (with the usual modification for $q=\infty$).
    \item The space $A_{p,q}^{s,\varphi}(\rd)$ denotes either $\btt(\rd)$ or $\ftt(\rd)$. Assume in addition that $\varphi$ satisfies \eqref{rintc} when $0<q<\infty$ and $A_{p,q}^{s,\varphi}(\rd)$ denotes $F_{p,q}^{s,\varphi}(\rd)$.
    \end{enumerate}
\end{definition}

\begin{example}
   \begin{enumerate}[\bfseries\upshape  (i)]
        \item If $\varphi:= 1$, then $B_{p,q}^{s,1}(\rd)$ and $F_{p,q}^{s,1}(\rd)$ become the Besov spaces $B_{p,q}^s(\rd)$ and the Triebel-Lizorkin spaces $F_{p,q}^s(\rd)$ respectively.
        \item If $\varphi(t):= t^{d\tau}$, where $\tau\geq 0$, then $B_{p,q}^{s,\varphi}(\rd)$ and $F_{p,q}^{s,\varphi}(\rd)$ become the Besov-type spaces $B_{p,q}^{s,\tau}(\rd)$ and the Triebel-Lizorkin-type spaces $F_{p,q}^{s,\tau}(\rd)$ respectively.
    \end{enumerate}
\end{example}

\subsection{Basic Properties}

\begin{proposition}\label{t3}
The definition of the spaces $A_{p,q}^{s,\varphi}(\rd)$ is independent of the choice of the dyadic resolution of unity $\theta$.
\end{proposition}

\begin{proof}
We prove the case of $\btt(\rd)$ spaces. For the case of $\ftt(\rd)$ spaces, it is the same if one uses (ii) of Theorem~\ref{thme} instead of (i) of Theorem~\ref{thme}.
Let $\theta=\{\theta_j\}_{j\in\no}$ and $\phi=\{\phi_j\}_{j\in\no}$ be arbitrary dyadic resolutions of unity.
If $\theta_{-1}:=0$, then $\theta_j=\theta_j\sum_{r=-1}^1\phi_{j+r}$ for $j\in\no$. Therefore,
\begin{align*}
    \mathcal{F}^{-1}\theta_j\mathcal{F}f=\sum_{j=-1}^1\mathcal{F}^{-1}\theta_j\mathcal{F}\mathcal{F}^{-1}\phi_{j+r}\mathcal{F}f.
\end{align*}
Choose $0<r<\min(p,q)$ and $\varkappa>\frac{d}{2}+\frac{d}{r}$.
If we replace $f_j$ and $\mu_j$ in (i) of Theorem~\ref{thme} by $\mathcal{F}^{-1}\phi_{j+r}\mathcal{F}f$ and $\theta_j$ respectively, we have
\begin{align*}
    &\|\{\mathcal{F}^{-1}\theta_j\mathcal{F}\mathcal{F}^{-1}\phi_{j+r}\mathcal{F}f\}_{j\in\no}\mid\lql\|\\
    &\qquad\lesssim\sup_j\|\theta_j(2^j\cdot)\mid H_2^\varkappa(\rd)\|
    \|\{\mathcal{F}^{-1}\phi_{j+r}\mathcal{F}f\}_{j\in\no}\mid\lql\|.
\end{align*}
By \eqref{3e1}, \eqref{3e2} and Definition~\ref{d3}, for $r=-1,0,1$, $j\in\no$,
\begin{align*}
  \|\{\mathcal{F}^{-1}\theta_j\phi_{j+r}\mathcal{F}f\}_{j\in\no}\mid\lql\|
  \lesssim\|\{\mathcal{F}^{-1}\phi_{j+r}\mathcal{F}f\}_{j\in\no}\mid\lql\|.
\end{align*}
Thus,
\begin{align*}
  \|\{\mathcal{F}^{-1}\theta_j\mathcal{F}f\}_{j\in\no}\mid\lql\|
  \lesssim\|\{\mathcal{F}^{-1}\phi_{j}\mathcal{F}f\}_{j\in\no}\mid\lql\|.
\end{align*}
Hence we finish the proof of Theorem~\ref{t3}.
\end{proof}

Next we establish the maximal inequalities. Recall that if $L\in\nat$, then $\mathcal{U}_L(\rd)$ denotes the collection of all functions $\theta=\{\theta_j\}_{j\in\no}\subset\sd$ with compact supports such that
\begin{align*}
    L(\theta):=\sup_{x\in\rd}|x|^L\sum_{|\alpha|\leq L}|\Dd^\alpha\theta_0(x)|
    +\sup_{x\in\rd\setminus\{0\},j\in\nat}(|x|^L+|x|^{-L})\sum_{|\alpha|\leq L}|\Dd^\alpha\theta_j(2^jx)|<\infty,
\end{align*}
where, here and hereafter, $\Dd^\alpha=\frac{\partial^{|\alpha|}}{\partial x_1^{\alpha_1}\cdots\partial x_d^{\alpha_d}}$
with $\alpha=(\alpha_1,\cdots,\alpha_d)$, $\alpha_i\in\nat\cup\{0\}$ and $|\alpha|=\sum_{j=1}^d\alpha_j$.

\begin{definition}(\cite[p.53]{t83})
Let $L\in\nat$, $\theta=\{\theta_j\}_{j\in\no}\in\mathcal{U}_L(\rd)$, $f\in\sdd$ and $a>0$,
then we define Peetre's maximal function
\begin{align}\label{nmfx}
    (\theta_j^\ast f)(x):=\sup_{y\in\rd}\frac{|(\mathcal{F}^{-1}\theta_j\mathcal{F}f)(x-y)|}{1+|2^jy|^a},\qquad x\in\rd,\, j\in\no.
\end{align}
\end{definition}

\begin{proposition}\label{up1}
Let $s\in\real$, $0<p<\infty$, $0<q\leq\infty$, $\varphi\in\Gp$ and $\theta=\{\theta_j\}_{j\in\no}$ be the above dyadic resolution of unity.
Let $a>0$ in \eqref{nmfx} be fixed, and let $L$ be a natural number with $L>|s|+3a+d+2$.
Assume that $\theta^\delta=\{\theta_j^\delta\}_{j\in\no}\in\mathcal{U}_L(\rd)$ with $0<\delta<1$.
  \begin{enumerate}[\bfseries\upshape  (i)]
    \item Then for all $f\in\sdd$,
    \begin{align*}
        \left\|\left\{2^{js}\sup_{0<\delta<1}(\theta_j^{\delta\ast}f)(\cdot)\right\}_{j\in\no}\mid\lql\right\|
        \lesssim\sup_{0<\delta<1}L(\theta^\delta)\|\{2^{js}(\theta_j^{\ast}f)(\cdot)\}_{j\in\no}\mid\lql\|.
    \end{align*}
    \item Assume in addition that $\varphi$ satisfies \eqref{rintc} when $0<q<\infty$. Then for all $f\in\sdd$,
    \begin{align*}
        \left\|\left\{2^{js}\sup_{0<\delta<1}(\theta_j^{\delta\ast}f)(\cdot)\right\}_{j\in\no}\mid\lpl\right\|
        \lesssim\sup_{0<\delta<1}L(\theta^\delta)\|\{2^{js}(\theta_j^{\ast}f)(\cdot)\}_{j\in\no}\mid\lpl\|.
    \end{align*}
\end{enumerate}
\end{proposition}

\begin{proof}
From the first step of the proof of Proposition 2.3.6 in \cite{t83}, we have
\begin{align*}
    2^{js}(\theta_j^{\delta\ast}f)(x)
    \lesssim L(\theta^\delta)\sum_{k=0}^\infty2^{(|s|+2a+N+\frac{d}{2}-L)|k-j|}2^{ks}(\theta_k^\ast f)(x),
\end{align*}
where $N$ is an even number with $a+\frac{d}{2}<N\leq a+\frac{d}{2}+2$ and the implicit constant in $\lesssim$ is independent of $f\in\sdd$, $x\in\rd$, $j$, $\delta$, $\theta$ and $\theta^\delta$.

Let $\gamma:=|s|+2a+N+\frac{d}{2}$ and let $L>\gamma$, then applying Proposition~\ref{ggl}, we obtain the desired results.
\end{proof}

\begin{remark}
From the above proof, we conclude that $L>|s|+2a+d/2+a+d/2+2=|s|+3a+d+2$.
\end{remark}

\begin{theorem}\label{prt1}
Let $s\in\real$, $0<p<\infty$, $0<q\leq\infty$, $\varphi\in\Gp$ and $\theta=\{\theta_j\}_{j\in\no}$ be the above dyadic resolution of unity.
Let $a>\frac{d}{\min\{p,q\}}$ and $L>|s|+3a+d+2$.
Assume that $\theta^\delta=\{\theta_j^\delta\}_{j\in\no}\in\mathcal{U}_L(\rd)$ with $0<\delta<1$,
  \begin{enumerate}[\bfseries\upshape  (i)]
    \item  Then for all $f\in\sdd$,
    \begin{align*}
         \left\|\left\{2^{js}\sup_{0<\delta<1}(\theta_j^{\delta\ast}f)(\cdot)\right\}_{j\in\no}\mid\lql\right\|
         \lesssim\sup_{0<\delta<1}L(\theta^\delta)\|f\mid\btt(\rd)\|.
    \end{align*}
     \item Assume in addition that $\varphi$ satisfies \eqref{rintc} when $0<q<\infty$. Then for all $f\in\sdd$,
    \begin{align*}
         \left\|\left\{2^{js}\sup_{0<\delta<1}(\theta_j^{\delta\ast}f)(\cdot)\right\}_{j\in\no}\mid\lpl\right\|
         \lesssim\sup_{0<\delta<1}L(\theta^\delta)\|f\mid\ftt(\rd)\|.
    \end{align*}
\end{enumerate}
\end{theorem}

\begin{proof}
By similarity, we only prove (ii). By (ii) of Proposition~\ref{up1}, it is sufficient to show that
\begin{align}\label{guocheng1}
    \|\{2^{js}(\theta_j^\ast f)(\cdot)\}_{j\in\no}\mid\lpl\|
    \lesssim\|\{2^{js}\mathcal{F}^{-1}(\theta_j\mathcal{F}f)\}_{j\in\no}\mid\lpl\|
\end{align}
if $f\in\ftt(\rd)$ and $a>d/\min\{p,q\}$ in \eqref{nmfx}.
However, \eqref{guocheng1} follows from \eqref{maxs2} with $D_j=2^{j}$ (cf. \eqref{3e1} and \eqref{3e2}), $d/r=a$
and $f_j=\mathcal{F}^{-1}(\theta_j\mathcal{F}f)$. This finishes the proof.
\end{proof}

Now we consider the lifting properties. If $\varkappa\in\real$, then the lift operator $I_\varkappa$ is defined by
$$
I_\varkappa f:=\mathcal{F}^{-1}(1+|x|^2)^\frac{\varkappa}{2}\mathcal{F}f,\qquad f\in\sdd.
$$

\begin{theorem}\label{lift}
Let $0<p<\infty$, $0<q\leq\infty$, $s\in\real$ and $\varphi\in\Gp$.
Assume in addition that $\varphi$ satisfies \eqref{rintc} when $0<q<\infty$ and $A_{p,q}^{s,\varphi}(\rd)=\ftt(\rd)$. Then
$$
I_\varkappa: A_{p,q}^{s,\varphi}(\rd)\rightarrow A_{p,q}^{s-\varkappa,\varphi}(\rd)
$$
is an isomorphism.
\end{theorem}

\begin{proof}
We prove the case of $\btt(\rd)$ spaces. The proof of the case of $\ftt(\rd)$ spaces is similar.
Let $\theta=\{\theta_j\}_{j\in\no}$ be the dyadic resolution of unity,
then $\phi=\{\phi_j\}_{j\in\no}\in\mathcal{U}_L(\rd)$, where $L$ is an arbitrary natural number
and $\phi_j(x)=2^{-j\varkappa}(1+x^2)^\frac{\varkappa}{2}\theta_j(x)$.
If $f\in\btt(\rd)$, then by (i) of Theorem~\ref{prt1}
and the estimate $|(\mathcal{F}^{-1}\phi_j\mathcal{F}f)(x)|\leq\phi_j^\ast f(x)$, we have
\begin{align*}
    \|I_\varkappa f\mid B_{p,q}^{s-\varkappa,\varphi}(\rd)\|
    &=\|\{2^{(s-\varkappa)}\mathcal{F}^{-1}(1+|\cdot|^2)^\frac{\varkappa}{2}\theta_j\mathcal{F}f\}_{j\in\no}\mid\lql\|\\
    &=\|\{2^{j s}\mathcal{F}^{-1}\phi_j\mathcal{F}f\}_{j\in\no}\mid\lql\|\\
    &\lesssim\|f\mid\btt(\rd)\|,
\end{align*}
which finishes the proof of Theorem~\ref{lift}.
\end{proof}

\subsection{Embedding Properties}

We start with some rather elementary embedding results.

\begin{proposition}\label{fdmt}
Let $s\in\real$, $0<p<\infty$, $0<q\leq\infty$ and $\varphi\in\Gp$.
Assume in addition that $\varphi$ satisfies \eqref{rintc} when $0<q<\infty$ and $A_{p,q}^{s,\varphi}(\rd)=\ftt(\rd)$.
  \begin{enumerate}[\bfseries\upshape  (i)]
    \item  If $q_1\leq q_2$, then
    \begin{align*}
    A_{p,q_1}^{s,\varphi}(\rd)\eb A_{p,q_2}^{s,\varphi}(\rd).
    \end{align*}
    \item For $\varepsilon>0$ and arbitrary $0<q_1, q_2\leq\infty$,
    \begin{align*}
    A_{p,q_1}^{s+\varepsilon,\varphi}(\rd)\eb A_{p,q_2}^{s,\varphi}(\rd).
    \end{align*}
    \item
    \begin{align*}
    B_{p,\min(p,q)}^{s,\varphi}(\rd)\eb F_{p,q}^{s,\varphi}(\rd)\eb B_{p,\max(p,q)}^{s,\varphi}(\rd).
    \end{align*}
\end{enumerate}
\end{proposition}

\begin{proof}
The properties (i), (ii) are simple corollaries of both the monotonicity of the $\ell^q$-norm and H\"older's inequality. Hence we omit the details.
We prove (iii). It is trivial that $B_{p,p}^{s,\varphi}(\rd)=F_{p,p}^{s,\varphi}(\rd)$.
If $0<q\leq p<\infty$, then $\ftt(\rd)\eb F_{p,p}^{s,\varphi}(\rd)=B_{p,p}^{s,\varphi}(\rd)=B_{p,\max(p,q)}^{s,\varphi}$.
Thus, it remains to verify that $\btt(\rd)\eb\ftt(\rd)$. To this end,
it suffices to show that  $\|\{f_j\}_{j\in\no}\mid\lpl\|\leq\|\{f_j\}_{j\in\no}\mid\lql\|$.
But this follows from the (generalized) Minkowski inequality. The proof of the case $q>p$ is similar.
Hence the proof is complete.
\end{proof}

\begin{theorem}\label{sas}
Let $s\in\real$, $0<p<\infty$, $0<q\leq\infty$ and $\varphi\in\Gp$. Assume in addition that $\varphi$ satisfies \eqref{rintc} when $0<q<\infty$ and $A_{p,q}^{s,\varphi}(\rd)=\ftt(\rd)$. Then
\begin{align*}
     \sd\eb A_{p,q}^{s,\varphi}(\rd)\eb\sdd.
\end{align*}
\end{theorem}

Before we prove Theorem~\ref{sas}, we need the following two lemmas. For $M\in\no$ and $\psi\in\sd$ we use the notation
\begin{align*}
\|\psi\mid   \mathcal{S}_M(\rd)\| = \sup_{\alpha\in\no^n, |\alpha|\leq M}\ \sup_{x\in \rd} |\Dd^\alpha \psi(x)| (1+|x|)^{\nd+M+|\alpha|}.
\end{align*}

\begin{lemma}{\rm{(\cite[Lemma 2.4]{ysy10})}}\label{sm}
Let $M\in\no$ and $\phi,\psi\in\sd$ with $\phi$ satisfying $\int_{\rd}x^{\gamma}\phi(x)\,\dint x=0$
for all multi-indices $\gamma\in\no^d$ satisfying $|\gamma|\leq M$.
Then there exists a positive constant $C$ such that for all $j\in\no$ and $x\in\rd$,
\begin{align*}
    |\phi_j\ast\psi(x)|\leq C\|\phi\mid{\mathcal{S}_{M+1}}(\rd)\|\|\psi\mid{\mathcal{S}_{M+1}}(\rd)\|\frac{2^{-jM}}{(1+|x|)^{d+M}}.
\end{align*}
\end{lemma}

\begin{lemma}\label{abs}
Let $0<p<\infty$, $0<q\leq\infty$ and $\varphi\in\Gp$.
Assume in addition that $\varphi$ satisfies \eqref{rintc} when $0<q<\infty$ and $A_{p,q}^{s,\varphi}(\rd)=\ftt(\rd)$.
Let
\begin{align}\label{nesc}
    \sum_{j=\jjp}^\infty\frac{\varphi(2^{-j})}{2^{j(s-\frac{d}{p})}}<\infty
\end{align}
be satisfied. Then
\begin{align*}
     A_{p,q}^{s,\varphi}(\rd)\eb B_{\infty,1}^0(\rd).
\end{align*}
In particular, for such $s$,
\begin{align*}
     A_{p,q}^{s,\varphi}(\rd)\eb B_{\infty,1}^0(\rd)\eb\ C(\rd)\eb\sdd,
\end{align*}
where $C(\rd)$ stands for the Banach space consisting of bounded uniformly continuous functions.
\end{lemma}

\begin{proof}
Let $f\in\btt(\rd)$, and let $0<r<\min(p,q)$. We first prove the following inequality
\begin{align}\label{zjgc}
    |\mathcal{F}^{-1}(\theta_j\mathcal{F}f)(x)|\lesssim\frac{\varphi(2^{-j})}{2^{j(s-\frac{d}{p})}}\cdot\|f\mid\btt(\rd)\|.
\end{align}
By appropriate substitution of variables, applying the inequality
$1+|2^j(y-\widetilde{z})|^{\frac{d}{r}}\lesssim(1+|2^j(y-x)|^\frac{d}{r})(1+|2^j(x-\widetilde{z})|^\frac{d}{r})\lesssim(1+|2^j(x-\widetilde{z})|^\frac{d}{r})$
for all $y$ with $|y-x|\lesssim 2^{-j}$, \eqref{asse1} and H\"older's inequality,
\begin{align*}
    |\mathcal{F}^{-1}(\theta_j\mathcal{F}f)(x)|^r
    &\leq \left[\sup_{z\in\rd}\frac{\mathcal{F}^{-1}(\theta_j\mathcal{F}f)(x-z)}{1+|2^{j}z|^\frac{d}{r}}\right]^r
    \sim\left[\sup_{\widetilde{z}\in\rd}\frac{|\mathcal{F}^{-1}(\theta_j\mathcal{F}f)(\widetilde{z})|}
    {1+|2^j(x-\widetilde{z})|^\frac{d}{r}}\right]^r\\
    &\lesssim \left[\sup_{\widetilde{z}\in\rd}\frac{|\mathcal{F}^{-1}(\theta_j\mathcal{F}f)(\widetilde{z})|}
    {1+|2^j(y-\widetilde{z})|^\frac{d}{r}}\right]^r
    \sim \left[\sup_{z\in\rd}\frac{|\mathcal{F}^{-1}(\theta_j\mathcal{F}f)(y-z)|}{1+|2^j(z)|^\frac{d}{r}}\right]^r\\
    &\lesssim\mhl(|\mathcal{F}^{-1}(\theta_j\mathcal{F}f)|^r)(y)\\
    &\lesssim 2^{j d}\int_{P(x,2^{-j})}\mhl(|\mathcal{F}^{-1}(\theta_j\mathcal{F}f)|^r)(y)\,\dint y\\
    &\lesssim 2^{j d}\left\{\int_{P(x,2^{-j})}\left[\mhl(|\mathcal{F}^{-1}(\theta_j\mathcal{F}f)|^r)(y)
    \right]^{\frac{p}{r}}\,\dint y\right\}^\frac{r}{p}\cdot\left(\int_{P(x,2^{-j})}1\,\dint y\right)^{1-\frac{r}{p}}\\
    &= 2^{j d \frac{r}{p}}\varphi(2^{-j})^r\frac{1}{\varphi(2^{-j})^r}\left\{\int_{P(x,2^{-j})}\left[
    \mhl(|\mathcal{F}^{-1}(\theta_j\mathcal{F}f)|^r)(y)\right]^{\frac{p}{r}}\,\dint y\right\}^\frac{r}{p}\\
    &= 2^{-j(s-\frac{d}{p})r}\varphi(2^{-j})^r\frac{2^{j s r}}{\varphi(2^{-j})^r}\left\{
    \int_{P(x,2^{-j})}\left[\mhl(|\mathcal{F}^{-1}(\theta_j\mathcal{F}f)|^r)(y)\right]^{\frac{p}{r}}\,\dint y\right\}^\frac{r}{p}.
\end{align*}
To obtain \eqref{zjgc}, it is sufficient to show that
\begin{align*}
    \frac{1}{\varphi(2^{-j})}2^{j s}\left\{\int_{P(x,2^{-j})}
    \left[\mhl(|\mathcal{F}^{-1}(\theta_j\mathcal{F}f)|^r)(y)\right]^{\frac{p}{r}}\,\dint y\right\}^\frac{1}{p}
    \lesssim\|f\mid B_{p,\infty}^{s,\varphi}(\rd)\|.
\end{align*}
Let $\nu\geq j$. Then by (i) of Theorem~\ref{mhlp}, we have
\begin{align*}
    &\frac{1}{\varphi(2^{-j})}2^{j s}\left\{\int_{P(x,2^{-j})}\left[
    \mhl(|\mathcal{F}^{-1}(\theta_j\mathcal{F}f)|^r)(y)\right]^{\frac{p}{r}}\,\dint y\right\}^\frac{1}{p}\\
    &\qquad \leq\frac{1}{\varphi(2^{-j})}\sup_{\nu\geq j}2^{\nu s}\left\{\int_{P(x,2^{-j})}\left[
    \mhl(|\mathcal{F}^{-1}(\theta_j\mathcal{F}f)|^r)(y)\right]^{\frac{p}{r}}\,\dint y\right\}^\frac{1}{p}\\
    &\qquad \lesssim\frac{1}{\varphi(2^{-j})}\sup_{\nu\geq j}2^{\nu s}\left\{
    \int_{P(x,2^{-j})}|\mathcal{F}^{-1}(\theta_j\mathcal{F}f)(y)|^p\,\dint y\right\}^\frac{1}{p}\\
    &\qquad \lesssim\sup_{j}\frac{1}{\varphi(2^{-j})}\sup_{\nu\geq j}2^{\nu s}\left\{
    \int_{P(x,2^{-j})}|\mathcal{F}^{-1}(\theta_j\mathcal{F}f)(y)|^p\,\dint y\right\}^\frac{1}{p}\\
    &\qquad \lesssim\|f\mid B_{p,\infty}^{s,\varphi}(\rd)\|.
\end{align*}
Thus, we obtain \eqref{zjgc} and then by \eqref{nesc}, we have
\begin{align*}
    \|f\mid B_{\infty,1}^0(\rd)\|\lesssim\|f\mid\btt(\rd)\|.
\end{align*}
This together with (iii) of Proposition~\ref{fdmt}, we obtain
\begin{align*}
    A_{p,q}^{s,\varphi}(\rd)\eb B_{\infty,1}^0(\rd)\eb C(\rd)\eb\sdd.
\end{align*}
\end{proof}

\begin{remark}
  \begin{enumerate}[\bfseries\upshape  (i)]
  \item
    Note that \eqref{nesc} is always satisfied when $s>\frac{d}{p}$ since any $\varphi\in \Gp$ satisfies $\varphi(2^{-j}) \leq \varphi(1)$, $j\geq 0$. So the interesting case is \eqref{nesc} with $s\leq \frac{d}{p}$.
        \item If $\varphi:=1$, then \eqref{nesc} reads as $s>\frac{d}{p}$, which means that $B_{p,q}^s(\rd)\eb B_{\infty,1}^0(\rd)\eb C(\rd)$ (cf. \cite{t83}).
        \item If $\varphi(t):= t^{d\tau}$ with $0\leq\tau<\infty$, then \eqref{nesc} reads as $s>\frac{d}{p}-d\tau$, which means that
        $B_{p,q}^{s,\tau}(\rd)\eb B_{\infty,1}^0(\rd)\eb C(\rd)$. However, since $B_{p,q}^{s,\tau}(\rd)\eb B_{\infty,\infty}^{s+d\tau-\frac{d}{p}}(\rd)$, and $B_{\infty,\infty}^{s+d\tau-\frac{d}{p}}(\rd)\eb B_{\infty,1}^0(\rd)\eb C(\rd)$ if $s+d\tau-\frac{d}{p}>0$ (cf. \cite{yy13b}), this is well-known.
    \end{enumerate}
\end{remark}

Now we are ready to give the proof of Theorem~\ref{sas}.

\begin{proof}[Proof of Theorem~\ref{sas}.]
We prove $\sd\eb\Att(\rd)$. Let $f\in\sd$ and $P$ be an arbitrary dyadic cube.
If $j_P>0$, applying Lemma~\ref{sm} with $M>\max(0, \frac{d}{p}-d, s+\frac{d}{p})$ and \eqref{gpi}, we have
\begin{align*}
    &\frac{1}{\varphi(\ell(P))}\left\{\sum_{j=j_P}^\infty\left[
    \int_P(2^{js}|\mathcal{F}^{-1}(\theta_j\mathcal{F}f)(x)|)^p
    \,\dint x\right]^\frac{q}{p}\right\}^\frac{1}{q}\\
    &\qquad=\frac{1}{\varphi(\ell(P))}\left\{\sum_{j=j_P}^\infty\left[
    \int_P(2^{js}|\mathcal{F}^{-1}\theta_j\ast f(x)|)^p
    \,\dint x\right]^\frac{q}{p}\right\}^\frac{1}{q}\\
    &\qquad\lesssim\|\mathcal{F}^{-1}\theta\mid\mathcal{S}_{M+1}(\rd)\|
    \|f\mid\mathcal{S}_{M+1}(\rd)\|\frac{1}{\varphi(\ell(P))}\left\{\sum_{j=j_P}^\infty\left[
    \int_P\frac{2^{j(s-M)p}}{(1+|x|)^{(d+M)p}}\,\dint x\right]^\frac{q}{p}\right\}^\frac{1}{q}\\
    &\qquad\lesssim\|\mathcal{F}^{-1}\theta\mid\mathcal{S}_{M+1}(\rd)\|\|f\mid\mathcal{S}_{M+1}(\rd)\|2^{j_P(s+\frac{d}{p}-M)}\\
    &\qquad\lesssim\|\mathcal{F}^{-1}\theta\mid\mathcal{S}_{M+1}(\rd)\|\|f\mid\mathcal{S}_{M+1}(\rd)\|.
\end{align*}
If $j_P\leq0$, then $\varphi(\ell(P))^{-1}\leq1$. Using Lemma~\ref{sm} again with $M>\max(0, \frac{d}{p}-d, s)$, we have
\begin{align*}
    &\frac{1}{\varphi(\ell(P))}\left\{\sum_{j=0}^\infty\left[
    \int_P(2^{js}|\mathcal{F}^{-1}(\theta_j\mathcal{F}f)(x)|)^p\,\dint x\right]^\frac{q}{p}\right\}^\frac{1}{q}\\
    &\qquad\lesssim\|\mathcal{F}^{-1}\theta\mid\mathcal{S}_{M+1}(\rd)\|
    \|f\mid\mathcal{S}_{M+1}(\rd)\|\left\{\sum_{j=0}^\infty\left[
    \int_P\frac{2^{j(s-M)p}}{(1+|x|)^{(d+M)p}}\,\dint x\right]^\frac{q}{p}\right\}^\frac{1}{q}\\
    &\qquad\lesssim\|\mathcal{F}^{-1}\theta\mid\mathcal{S}_{M+1}(\rd)\|\|f\mid\mathcal{S}_{M+1}(\rd)\|.
\end{align*}
Thus, $\|f\mid\btt(\rd)\|\lesssim\|f\mid\mathcal{S}_{M+1}(\rd)\|$. This together with (iii) of Proposition~\ref{fdmt}, we obtain $\sd\eb\Att(\rd)$.

Finally, we prove $\Att(\rd)\eb\sdd$. From Theorem~\ref{lift} and $A_{p,q}^{s,\varphi}(\rd)\eb B_{p,\infty}^{s,\varphi}(\rd)$ in the sense of a continuous embedding, we only have to prove
\begin{align*}
    B_{p,\infty}^{s,\varphi}(\rd)\eb\sdd
\end{align*}
for $s\gg1$, which is already done in Lemma~\ref{abs}. Hence we finish the proof.
\end{proof}

Now we come to some new Sobolev-type embeddings.

\begin{theorem}\label{ste}
Let $0<p<\infty$, $0<r,q\leq\infty$, $\varphi\in\Gp$ and $-\infty<s_2<s_1<\infty$. If $0<p_1<p_2<\infty$ such that
\begin{align*}
    s_1-\frac{d}{p_1}=s_2-\frac{d}{p_2},
\end{align*}
  \begin{enumerate}[\bfseries \upshape (i)]
    \item  then
    \begin{align*}
        B_{p_1,q}^{s_1,\varphi}(\rd)\eb B_{p_2,q}^{s_2,\varphi}(\rd);
    \end{align*}
    \item  assume in addition that $\varphi$ satisfies \eqref{rintc} when $0<q<\infty$, then
    \begin{align*}
        F_{p_1,r}^{s_1,\varphi}(\rd)\eb F_{p_2,q}^{s_2,\varphi}(\rd).
    \end{align*}
\end{enumerate}
\end{theorem}

The following lemma is crucial to obtain Theorem~\ref{ste}.

\begin{lemma}\label{lpp}
Let $0<p_1<p_2\leq\infty$. Then for $f\in\mathcal{S}^{\Omega_Q}(\rd)$ and all cubes $Q(x, 2^{-j})\in\mq$, $j\in\mathbb{Z}$,
    \begin{align}\label{lppe}
        \|f\mid L^{p_2}(Q(x, 2^{-j}))\|\lesssim 2^{j(\frac{d}{p_1}-\frac{d}{p_2})}\|f\mid L^{p_1}(Q(x, 2^{-j}))\|,
    \end{align}
where $\mathcal{S}^{\Omega_Q}(\rd)$ stands for the set of all distributions whose Fourier transform is contained in the closure $\overline{Q(x, 2^{-j})}$.
\end{lemma}

\begin{proof}
We first prove that
\begin{align}\label{hpe}
    \|f\mid L^{\infty}(Q(x, 2^{-j}))\|\lesssim 2^{j\frac{d}{p_1}}\|f\mid L^{p_1}(Q(x, 2^{-j}))\|,\qquad 0<p_1\leq\infty.
\end{align}
Here we follow similar arguments presented in \cite[Sect.~1.3.2]{t83}.
If $1\leq p_1\leq\infty$, then by \cite[(1.3.1.3)]{t83} and by H\"older's inequality, we have
\begin{align*}
    |f(x)|&\leq c\ 2^{jd}\int_{Q(x, 2^{-j})}|f(y)|\,\dint y
    \leq c\ 2^{jd}2^{-jd(1-\frac{1}{p_1})}\left(\int_{Q(x, 2^{-j})}|f(y)|^{p_1}\,\dint y\right)^\frac{1}{p_1}
    = c \ 2^{j\frac{d}{p_1}}\|f\mid L^{p_1}(Q(x, 2^{-j}))\|.
\end{align*}
If $0<p_1<1$, we have
\begin{align*}
    |f(x)|&\leq c\ 2^{jd}\int_{\rd}|\chi_{Q(x, 2^{-j})}f(y)|\,\dint y
    \lesssim 2^{jd}\left(\sup_{y\in Q(x,2^{-j})}|f(y)|\right)^{1-p_1}\int_{Q(x, 2^{-j})}|f(y)|^{p_1}\,\dint y.
\end{align*}
Taking the supremum with respect to $x\in Q(x, 2^{-j})$, we obtain \eqref{hpe}. Thus, by \eqref{hpe} and $p_1<p_2$, we have
\begin{align*}
    \|f\mid L^{p_2}(Q(x, 2^{-j}))\|
    &\lesssim \left(\sup_{y\in Q(x,2^{-j})}|f(y)|\right)^{1-\frac{p_1}{p_2}}
    \left(\int_{Q(x, 2^{-j})}|f(y)|^{p_1}\,\dint y\right)^\frac{1}{p_2}\\
    &\lesssim2^{j(\frac{d}{p_1}-\frac{d}{p_2})}\|f\mid L^{p_1}(Q(x, 2^{-j}))\|,
\end{align*}
which completes the proof.
\end{proof}

\begin{proof}[Proof of Theorem~\ref{ste}]
We first prove (i). Let $f\in B_{p_1,q}^{s_1,\varphi}(\rd)$. From Lemma~\ref{lpp}, we have
\begin{align*}
    &\frac{1}{\varphi(\ell(P))}\left\{\sum_{j=\jjp}^\infty2^{j s_2 q}\left(\int_P
    |\mathcal{F}^{-1}(\theta_j\mathcal{F}f)(x)|^{p_2}\,\dint x\right)^\frac{q}{p_2}\right\}^\frac{1}{q}\\
    &\qquad\lesssim\frac{1}{\varphi(\ell(P))}\left\{\sum_{j=\jjp}^\infty2^{j(s_2-\frac{d}{p_2}+\frac{d}{p_1})q}
    \left(\int_P|\mathcal{F}^{-1}(\theta_j\mathcal{F}f)(x)|^{p_1}\,\dint x\right)^\frac{q}{p_1}\right\}^\frac{1}{q}\\
    &\qquad\lesssim\|f\mid B_{p_1,q}^{s_1,\varphi}(\rd)\|.
\end{align*}

Next we prove (ii). By Theorem~\ref{lift}, we may assume that $s_1=0$. Then $s_2<0$.
Furthermore, we assume that $r=\infty$ and $0<q<1$.
Let $f\in F_{p_1,\infty}^{0,\varphi}(\rd)$ with $\|f\mid F_{p_1,\infty}^{0,\varphi}(\rd)\|=1$.
Then for all $P\in\mq$ with $\ell(P)=2^{-j}\leq 1$, by Lemma~\ref{lpp} and $\varphi\in\Gp$, we have
\begin{align*}
    |\fuff(x)|\lesssim 2^{\frac{jd}{p_1}}\left(\int_P|\fuff(y)|^{p_1}\,\dint y\right)^\frac{1}{p_1}
    \lesssim 2^{\frac{jd}{p_1}}\varphi(\ell(P))\|f\mid F_{p_1,\infty}^{0,\varphi}(\rd)\|
    \lesssim 2^{\frac{jd}{p_1}}.
\end{align*}
Let $J\in\no$. If $J\geq\jjp$, then
\begin{align*}
    \left(\sum_{j=\jjp}^J|2^{js_2}\fuff(x)|^q\right)^\frac{1}{q}
    \leq C\left(\sum_{j=\jjp}^J2^{j\frac{d}{p_2}q}\right)^\frac{1}{q}
    \leq C 2^{J\frac{d}{p_2}}\leq t,
\end{align*}
if $t\sim C 2^{J\frac{d}{p_2}}$, where $C$ is independent of $J$, and
\begin{align*}
    \left(\sum_{j=J+1}^\infty|2^{js_2}\fuff(x)|^q\right)^\frac{1}{q}
    \leq C 2^{Js_2}\sup_{j\in\no}|\fuff(x)|
    \leq C t^{1-\frac{p_2}{p_1}}\sup_{j\in\no}|\fuff(x)|.
\end{align*}
Combining the above two estimates, we have
\begin{align*}
    \|f\mid F_{p_2,q}^{s_2,\varphi}(\rd)\|^{p_2}
    &=\sup_{P\in\mq}\frac{p_2}{\varphi(\ell(P))}\int_0^\infty t^{p_2-1}\left|\left\{\left(
    \sum_{j=\jjp}^\infty |2^{js_2}\fuff|^q\right)^\frac{1}{q}>t\right\}\right|\,\dint t\\
    &\leq \sup_{P\in\mq}\frac{p_2}{\varphi(\ell(P))}\int_0^\infty t^{p_2-1}\left|\left\{
    \sup_{j\in\no}|\fuff|>Ct^{\frac{p_2}{p_1}}\right\}\right|\,\dint t\\
    &\lesssim \sup_{P\in\mq}\frac{p_1}{\varphi(\ell(P))}\int_0^\infty t^{p_1-1}\left|\left\{
    \sup_{j\in\no}|\fuff|>t\right\}\right|\,\dint t\\
    &\lesssim 1.
\end{align*}
Assume now $J<\jjp$. Note that
\begin{align*}
     \left(\sum_{j=\jjp}^\infty|2^{js_2}\fuff(x)|^q\right)^\frac{1}{q}\leq  \left(\sum_{j=J+1}^\infty|2^{js_2}\fuff(x)|^q\right)^\frac{1}{q}
\end{align*}
and by the same argument as above, we also have $\|f\mid F_{p_2,q}^{s_2,\varphi}(\rd)\|\lesssim 1$, which finishes the proof of Theorem~\ref{ste}.
\end{proof}

\begin{remark}
     Since we consider the above Sobolev-type embeddings with fixed $\varphi$, Theorem~\ref{ste} corresponds to \cite[Theorem 2.7.1]{t83} when $\varphi:=1$ and to \cite[Corollary 2.2]{ysy10} when $\varphi(t):=t^{d\tau}$ with $0\leq\tau<\infty$.
\end{remark}

\begin{proposition}\label{tt5}
    Let $s\in\real$, $0<p<\infty$, $0<q\leq\infty$ and $\varphi\in\Gp$. Assume in addition that $\varphi$ satisfies \eqref{rintc} when $0<q<\infty$ and $A_{p,q}^{s,\varphi}(\rd)$ denotes $F_{p,q}^{s,\varphi}(\rd)$. Then
    \begin{align*}
        A_{p,q}^{s,\varphi}(\rd)\eb B_{\infty,\infty}^{s-\frac{d}{p}}(\rd).
    \end{align*}
    In particular,
    \begin{align*}
        A_{p,q}^{s,\varphi}(\rd)\eb \mathcal{C}^{s-\frac{d}{p}}(\rd),\qquad\text{if}\quad s>\frac{d}{p},
    \end{align*}
    where $\mathcal{C}^\varkappa(\rd)$ with $\varkappa>0$ denotes the H\"older-Zygmund spaces.
\end{proposition}

\begin{proof}
    Let $\phi\in\rd$ such that $\mathcal{F}\phi:=1$ on $\{\xi\in\rd:\,|\xi|\leq 1\}$. Then by \cite[(1.3.1.3)]{t83}, we have
    \begin{align*}
        \fuff=(\mathcal{F}\theta_j)\ast f\ast \phi_j.
    \end{align*}
    Thus, for all $j\in\no$ and $x\in\rd$, from \cite[(2.11)]{fj85}, we see that
    \begin{align*}
        |\fuff(x)|&\leq\sum_{m\in\zd}\int_{Q_{j,m}}|(\mathcal{F}\theta_j)\ast f(y)||\phi_j(x-y)|\,\dint y\\
        &\leq\sum_{m\in\zd}\sup_{z\in Q_{j,m}}|(\mathcal{F}\theta_j)\ast f(z)|\int_{Q_{j,m}}|\phi_j(x-y)|\,\dint y\\
        &\lesssim\sum_{m\in\zd}2^{j\frac{d}{p}}\left[\sum_{\ell\in\zd}(1+|\ell|)^{d-1}
        \int_{Q_{j,k+l}}|(\mathcal{F}\theta_j)\ast f(z)|^p\,\dint z\right]^\frac{1}{p}
        \times\int_{Q_{j,m}}|\phi_j(x-y)|\,\dint y\\
        &\lesssim2^{j(\frac{d}{p}-s)}\varphi(2^{-j})\|f\mid A_{p,q}^{s,\varphi}(\rd)\|
        \times\int_{Q_{j,m}}|\phi_j(x-y)|\,\dint y\\
        &\lesssim2^{j(\frac{d}{p}-s)}\|f\mid A_{p,q}^{s,\varphi}(\rd)\|,
    \end{align*}
    which finishes the proof of Proposition~\ref{tt5}.
\end{proof}

\begin{remark}
Note that for $s>\frac{d}{p}$, Proposition~\ref{tt5} together with Proposition~\ref{fdmt}(ii) yields another proof for Lemma~\ref{abs}. But as we already discussed there, the assumption $s>\frac{d}{p}$ is not needed for the statement of Lemma~\ref{abs}, but only the weaker one \eqref{nesc}.
\end{remark}

Now we would like to extend the result \cite[Prop.~2.7, p.47]{ysy10}. Let $L^p_\varphi(\rd)$, $0<p<\infty$, $\varphi\in\Gp$, be the set of all locally $p$-integrable functions $f:\rd\to\mathbb{C}$ such that
\[
\|f \mid L^p_\varphi(\rd)\| = \sup_{P\in \mq, |P|\geq 1} \frac{1}{\varphi(\ell(P))} \left(\int_P |f(x)|^p \dint x\right)^{\frac1p}.
  \]
  Plainly, $L^p_\varphi(\rd)= L^p(\rd)$ when $\varphi\equiv 1$, and $L^p_\varphi(\rd)=L^p_\tau(\rd)$ from \cite[eq. (2.18)]{ysy10}, when $\varphi(t)=t^{d\tau}$, $\tau\geq 0$. Obviously, if $\varphi_1, \varphi_2\in \Gp$ with $\varphi_1(t)\leq \varphi_2(t)$, $t\geq 1$, then $L^p_{\varphi_1}(\rd)\hookrightarrow L^p_{\varphi_2}(\rd)$. Moreover,  $L^p_\varphi(\rd)\subset \sdd$ for $1\leq p<\infty$, as can be seen as follows: let $\phi\in\sd$, then, for arbitrary $M\in\nat$,
  \begin{align*}
    |\langle f,\phi\rangle| & \leq \int_{\rd} |f(x)| |\phi(x)| \dint x \quad  \leq \sum_{k\in\zd} \left(\int_{Q_{0,k}} |f(x)|^p \dint x\right)^{\frac1p}  \left(\int_{Q_{0,k}} |\phi(x)|^{p'} \dint x\right)^{\frac{1}{p'}}  \\
    &\leq \varphi(1) \ \|f\mid L^p_\varphi(\rd)\| \|\phi\mid \mathcal{S}_M(\rd)\| \sum_{k\in\zd} (1+|k|)^{-\nd-M} \ \lesssim \ \| f \mid L^p_\varphi(\rd)\| \|\phi\mid \mathcal{S}_M(\rd)\|.
  \end{align*}

  Now we are ready to give the counterpart of \cite[Prop.~2.7, p.47]{ysy10}.

\begin{proposition}
    Let $\varphi\in\Gp$ and $\varphi$ satisfy \eqref{rintc} for $F$-spaces when $0<q<\infty$.
  \begin{enumerate}[\bfseries\upshape (i)]
        \item  Let $1\leq p< \infty$. Then
        \begin{align*}
            B_{p,1}^{0,\varphi}(\rd)\eb F_{p,1}^{0,\varphi}(\rd)\eb L_{\varphi}^p(\rd).
        \end{align*}
        \item  Let $0<p<1$. Then
        \begin{align*}
            B_{p,p}^{\sigma_p,\varphi}(\rd)=F_{p,p}^{\sigma_p,\varphi}(\rd)
            \eb \left(L_\varphi^1(\rd)\cap L_\varphi^p(\rd)\right).
        \end{align*}
    \end{enumerate}
\end{proposition}

\begin{proof}
    We first prove (i). Note that the left-hand embedding in (i) follows from Proposition~\ref{fdmt} (iii), so we only need to prove $F_{p,1}^{0,\varphi}(\rd)\eb L_{\varphi}^p(\rd)$.
    Let $\{\theta_j\}_{j\in\no}$ be the above dyadic resolution of unity. Then, for each dyadic cube $P$,
    \begin{align*}
        f=\sum_{j=0}^\infty\theta_j\ast f
    \end{align*}
    in the sense of $L^p(P)$. Thus, if $\ell(P)\geq 1$, we have
    \begin{align*}
        \frac{1}{\varphi(\ell(P))}\|f\mid L^p(P)\|
        \leq \frac{1}{\varphi(\ell(P))}\left[\int_P\left(\sum_{j=0}^\infty|\theta_j\ast f(x)|\right)^p
        \,\dint x\right]^\frac{1}{p}
        \leq\|f\mid F_{p,1}^{0,\varphi}(\rd)\|.
    \end{align*}
    Taking the supremum over all cubes $P$, the right-hand side of (i) is verified.

    Next we prove (ii). From Proposition~\ref{fdmt}, Theorem~\ref{ste} and the above proof of (i), we conclude that
    \begin{align*}
        F_{p,p}^{\sigma_p,\varphi}(\rd) \eb F_{p,1}^{0,\varphi}(\rd)   \eb L_\varphi^1(\rd).
    \end{align*}
    Thus, we obtain
    \begin{align*}
        f=\sum_{j=0}^\infty\theta_j\ast f \qquad \text{in}\quad L^1(P).
    \end{align*}
    H\"older's inequality yields that this identity takes place in $L^p(P)$ as well. Then, if $\ell(P)\geq 1$, we have
    \begin{align*}
        \frac{1}{\varphi(\ell(P))}\|f\mid L^p(P)\|
        &\leq \frac{1}{\varphi(\ell(P))}\left[\int_P\left(\sum_{j=0}^\infty|\theta_j\ast f(x)|\right)^p
        \,\dint x\right]^\frac{1}{p}\\
        &\leq \frac{1}{\varphi(\ell(P))}
        \left[\int_P\sum_{j=0}^\infty|\theta_j\ast f(x)|^p\,\dint x\right]^\frac{1}{p}\\
        &\leq\|f\mid F_{p,1}^{0,\varphi}(\rd)\|.
    \end{align*}
    Hence we finish the proof.
\end{proof}

\begin{proposition}
    Let $s\in\real$, $0<p_2\leq p_1<\infty$, $0<q\leq \infty$ and $\varphi_1\in\mathcal{G}_{p_1}$. Assume that $\varphi_1$ satisfies \eqref{rintc} when $0<q<\infty$ and $A_{p_1,q}^{s,\varphi_1}(\rd)=F_{p_1,q}^{s,\varphi_1}(\rd)$. Then
    \begin{align*}
        A_{p_1,q}^{s,\varphi_1}(\rd)\eb A_{p_2,q}^{s,\varphi_2}(\rd)
    \end{align*}
    with $\varphi_2(t)=\varphi_1(t)t^{d(\frac{1}{p_2}-\frac{1}{p_1})}$, $t>0$.
\end{proposition}

\begin{proof}
    It is easy to see that $\varphi_2$ belongs to $\mathcal{G}_{p_2}$ and \eqref{rintc} holds as well. Then H\"older's inequality yields the desired results.
\end{proof}

\section{Atomic decomposition}\label{sec-atoms}

The final section gives the atomic decomposition of $B_{p,q}^{s,\varphi}(\rd)$ and $F_{p,q}^{s,\varphi}(\rd)$. We start with the corresponding sequence spaces of these spaces.

\begin{definition}
    Let $s\in\real$, $0<p<\infty$, $0<q\leq\infty$ and $\varphi\in\Gp$.
       \begin{enumerate}[\bfseries\upshape  (i)]
    \item  The sequence space $\btts(\rd)$ is defined to be the set of all sequences $\lz:=\{\lzjm\}_{j\in\no,m\in\zd}$ such that
        \begin{align*}
            \|\lz\mid\btts(\rd)\|:=
            \sup_{P\in\mq}\frac{1}{\varphi(\ell(P))}
            \left\{\sum_{j=\jjp}^\infty \left[\sum_{\substack{m\in\zd\\\qjm\subset P}} (2^{j (s-\frac{d}{p})}
            |\lzjm|)^p\right]^\frac{q}{p}\right\}^\frac{1}{q}
        \end{align*}
        is finite (with the usual modification for $q=\infty$).
        \item  Assume in addition that $\varphi$ satisfies \eqref{rintc} when $0<q<\infty$. The sequence space $\ftts(\rd)$ is defined to be the set of all sequences $\lz:=\{\lzjm\}_{j\in\no,m\in\zd}$ such that
            \begin{align*}
            \|\lz\mid\ftts(\rd)\|:=\sup_{P\in\mq}\frac{1}{\varphi(\ell(P))}
            \left\{\int_P\left[\sum_{j=\jjp}^\infty\sum_{\substack{m\in\zd\\\qjm\subset P}}
            (2^{j s}|\lzjm|\kjm(x))^q\right]^\frac{p}{q}\,\mathrm{d}x\right\}^\frac{1}{p}
        \end{align*}
        is finite (with the usual modification for $q=\infty$).
        \item The sequence space $a_{p,q}^{s,\varphi}(\rd)$ denotes either $\btts(\rd)$ or $\ftts(\rd)$. Assume in addition that $\varphi$ satisfies \eqref{rintc} when $0<q<\infty$ and $a_{p,q}^{s,\varphi}(\rd)$ denotes $f_{p,q}^{s,\varphi}(\rd)$.
    \end{enumerate}
\end{definition}

\begin{remark}
\begin{enumerate}[\bfseries\upshape  (i)]
        \item  In the present paper, the $L^2$-normalized indicator $|\qjm|^{-1/2}\kjm$ is not considered for the time being.
        \item  If $\varphi(t):=t^{d\tau}$, $\tau\geq 0$, then the sequence spaces $\btts(\rd)$ and $\ftts(\rd)$ become, respectively, the spaces $b_{p,q}^{s,\tau}(\rd)$ and $f_{p,q}^{s,\tau}(\rd)$ in \cite[Definition 2.2]{ysy10}.
    \end{enumerate}
\end{remark}

Next we introduce the atoms.

\begin{definition}
    Let $c>1$, $L\in\no\cup\{-1\}$ and $K\in\no$. A $C^K$-function $\ajm:\rd\rightarrow\mathbb{C}$ is said to be a $(K,L,c)$-atom supported near $\qjm$ with $j\in\no$ and $m\in\zd$, if
    \begin{align}
     2^{-j |\alpha|}|\Dd^\alpha \ajm(x)|\leq \chi_{c \qjm}(x)  \label{ac2}
    \end{align}
    for all $x\in\rd$ and for all $\alpha\in\no^d$ with $|\alpha|\leq K$ and
    \begin{align}\label{ac3}
        \int_{\rd}x^\beta \ajm(x)\,\dint  x=0
    \end{align}
    for all $\beta\in\no^d$ with $|\beta|\leq L$ when $L\geq 0$ and $j\in\nat$.
\end{definition}

\begin{remark}
   \begin{enumerate}[\bfseries\upshape  (i)]
        \item  If $L=-1$, then \eqref{ac3} means that there are no vanishing moment conditions required.
        \item  Similar to \cite[p.59]{ysy10}, an atom for $\Att(\rd)$ supported near $Q_{j,m}$ has vanishing moment conditions only when $\ell(\qjm)<1$, that is, for $j\in\nat$.
    \end{enumerate}
\end{remark}

We now state our main theorem in this section.

\begin{theorem}\label{adbf}
    Let $s\in\real$, $0<p<\infty$, $0<q\leq\infty$ and $\varphi\in\Gp$. Let also $c>1$, $L\in\no\cup\{-1\}$ and $K\in\no$. Assume that
    \begin{align*}
\begin{cases}
    K\geq \whole{1+s}_+,\quad L\geq \max(-1,\whole{\sigma_p-s}),&\qquad\text{if}\quad A_{p,q}^{s,\varphi}(\mathbb{R}^d)=B_{p,q}^{s,\varphi}(\mathbb{R}^d);         \\
    K\geq \whole{1+s}_+,\quad L\geq \max(-1,\whole{\sigma_{p,q}-s}),&\qquad\text{if}\quad A_{p,q}^{s,\varphi}(\mathbb{R}^d)=F_{p,q}^{s,\varphi}(\mathbb{R}^d).
\end{cases}
    \end{align*}
    Assume in addition that $\varphi$ satisfies \eqref{rintc} when $0<q<\infty$ and $\Att(\rd)=\ftt(\rd)$, $\att(\rd)=\ftts(\rd)$.
 \begin{enumerate}[\bfseries\upshape  (i)]
        \item  Let $f\in\Att(\rd)$. Then there exist a family $\{a_{Q_{j,m}}\}_{j\in\no,m\in\zd}$ of $(K,L,c)$-atoms and a sequence $\lz=\{\lzjm\}_{j\in\no,m\in\zd}\in\att(\rd)$ such that
        \begin{align}\label{afc1}
            f=\sum_{j=0}^\infty\sum_{m\in\zd}\lzjm \ajm\quad\text{in}\quad\sdd
        \end{align}
        and
        \begin{align}
            \|\lz\mid\att(\rd)\|\lesssim\|f\mid\Att(\rd)\|.
        \end{align}
        \item  Let $\{\ajm\}_{j\in\no,m\in\zd}$ be a family of $(K,L,c)$-atoms and
        $\lz=\{\lzjm\}_{j\in\no,m\in\zd}\in\att(\rd)$. Then
        \begin{align*}
            f=\sum_{j=0}^\infty\sum_{m\in\zd}\lzjm \ajm
        \end{align*}
        converges in $\sdd$ and belongs to $\Att(\rd)$. Furthermore,
        \begin{align}
            \|f\mid\Att(\rd)\|\lesssim\|\lz\mid\att(\rd)\|.
        \end{align}
    \end{enumerate}
\end{theorem}

Before we prove our main theorem, we need some technical lemmas.
First of all, we recall the Calder\'on reproducing formula introduced in \cite{fj90}.
Note that the dyadic resolution of unity $\theta=\{\theta_j\}_{j\in\no}$ introduced in Section 2 (cf. \eqref{3e1} and \eqref{3e2}) is equivalent to $\mathcal{F}\varphi_\nu$ and $\mathcal{F}\psi_\nu$ in \cite[p.45]{fj90}.
Let $\theta$ satisfies \eqref{3e1} and \eqref{3e2}. By \cite[pp.130-131]{fj90}, there exists a function $\upsilon\in\sd$ which satisfies \eqref{3e1} and \eqref{3e2} such that for all $\xi\in\rd$,
\begin{align}\label{crf1}
    \sum_{j=0}^\infty\widetilde{\theta}(2^{-j}\xi)\upsilon(2^{-j}\xi)=1,
\end{align}
where $\widetilde{\theta}(x):=\overline{\theta(-x)}$, $x\in\rd$. Now we have the following Calder\'on reproducing formula.
\begin{lemma}{\rm{\cite[(12.4)]{fj90}}}\label{crf}
    Let $\theta,\,\upsilon\in\sd$ satisfy \eqref{3e1} and \eqref{3e2} such that \eqref{crf1} holds. Then for all $f\in\sdd$,
    \begin{align}\label{crf2}
        f&=\sum_{j=0}^\infty 2^{-j d}\sum_{m\in\zd}
        \mathcal{F}^{-1}\widetilde{\theta_j}\ast f(2^{-j}m)\mathcal{F}^{-1}\upsilon_j(\cdot-2^{-j}m)\notag\\
        &=\sum_{j=0}^\infty\sum_{m\in\zd}\langle f,\,\mathcal{F}^{-1}\theta_{\qjm}\rangle \mathcal{F}^{-1}\upsilon_{\qjm}
    \end{align}
    in $\sdd$.
\end{lemma}

Let $0<r\leq\infty$ and let $\delta>0$ be fixed. For a sequence $\lz:=\{\lzjm\}_{j\in\no,m\in\zd}$, set
\begin{align*}
    (\lz^\ast_{r,\delta})_{\qjm}:=\left(\sum_{R_{j,k}\in\mq,\,k\in\zd}   \frac{|\lz_{R_{j,k}}|^r}{(1+\ell(R_{j,k})^{-1}|x_{R_{j,k}}-x_{Q_{j,m}}|)^\delta}\right)^{\frac{1}{r}},\quad \qjm\in\mq,\,j\in\no,\,m\in\zd.
\end{align*}
and $\lz^\ast_{r,\delta}:=\{(\lz^\ast_{r,\delta})_{\qjm}\}_{j\in\no,m\in\zd}$.

\begin{lemma}\label{el2}
Let $s\in\real$, $0<p<\infty$, $0<q\leq\infty$, $\delta>d$ and $\varphi\in\Gp$. Assume in addition that $\varphi$ satisfies \eqref{rintc} when $0<q<\infty$ and $a_{p,q}^{s,\varphi}(\rd)$ denotes $f_{p,q}^{s,\varphi}(\rd)$. Then for all $\lz\in\att(\rd)$,
$$
\|\lz\mid\att(\rd)\|\leq\|\lz^\ast_{\min(p,q),\delta}\mid\att(\rd)\|\lesssim\|\lz\mid\att(\rd)\|.
$$
\end{lemma}

\begin{proof}
Using $\varphi(\ell(P))$ to replace $|P|^\tau$ in the proof of \cite[Lemma 2.8]{ysy10}, we obtain Lemma~\ref{el2}. For reader's convenience, we give the details.

Let $\lz\in\att(\rd)$. It is easily seen that $|\lzjm|\leq(\lz^\ast_{\min(p,q),\delta})_{\qjm}$ holds for all dyadic cubes $\qjm$ with $j\in\no,\,m\in\zd$. Thus we have $\|\lz\mid\att(\rd)\|\leq\|\lz^\ast_{\min(p,q),\delta}\mid\att(\rd)\|$.

Next we prove the converse. Let a dyadic cube $P$ be fixed. For all dyadic cubes $\qjm$ with $j\in\no,\,m\in\zd$,
let $\omega_{\qjm}:=\lzjm$ if $\qjm\subset 3P$ and $\omega_{\qjm}=0$ otherwise,
and let $\mu_{\qjm}:=\lz_{\qjm}-\omega_{\qjm}$. Set $\omega:=\{\omega_{\qjm}\}_{j\in\no,m\in\zd}$
and $\mu:=\{\mu_{\qjm}\}_{j\in\no,m\in\zd}$. Then for all such $\qjm$, we have
\begin{align}\label{eel2}
    (\lz^\ast_{\min(p,q),\delta})_{\qjm}^{\min(p,q)}=(\omega^\ast_{\min(p,q),\delta})_{\qjm}^{\min(p,q)}+
    (\mu^\ast_{\min(p,q),\delta})_{\qjm}^{\min(p,q)}.
\end{align}
Applying \cite[Lemma 2.3]{fj90}, we have
\begin{align*}
    \mathrm{I}_P&:=\frac{1}{\varphi(\ell(P))}
            \left\{\sum_{j=\jjp}^\infty\left[\sum_{\substack{m\in\zd\\\qjm\subset P}}
            (2^{j (s-d)}(\omega^\ast_{\min(p,q),\delta})_{\qjm})^p\right]^\frac{q}{p}\right\}^\frac{1}{q}\\
        &\leq \frac{1}{\varphi(\ell(P))}\left\|\omega^\ast_{\min(p,q),\delta}\mid b_{p,q}^s(\rd)\right\|\\
        &\lesssim \frac{1}{\varphi(\ell(P))}\|\omega\mid b_{p,q}^s(\rd)\|\\
        &\lesssim \|\lz\mid b_{p,q}^{s,\varphi}(\rd)\|,
\end{align*}
where $b_{p,q}^s(\rd)$ is the corresponding sequence space for the Besov space $B_{p,q}^s(\rd)$.
Similarly, we have
\begin{align*}
\widetilde{\mathrm{I}}_P:=\frac{1}{\varphi(\ell(P))}
            \left\{\int_P\left[\sum_{j=\jjp}^\infty\sum_{\substack{m\in\zd\\\qjm\subset P}}
            \left(2^{j s}(\omega^\ast_{\min(p,q),\delta})_{\qjm}\kjm(x)\right)^q\right]^\frac{p}{q}\,\dint x\right\}^\frac{1}{p}\lesssim\|\lz\mid f_{p,q}^{s,\varphi}(\rd)\|.
\end{align*}

On the other hand, let $\qjm\subset P$ be a dyadic cube with side length at most 1.
Then $\ell(\qjm)=2^{-i}\ell(P)$ for some nonnegative integer $i\geq\max(-j_P,0)=-\min(j_P,0)$.
Suppose $\widetilde{\qjm}$ is any dyadic cube with $\ell(\widetilde{\qjm})=\ell(\qjm)=2^{-i}\ell(P)$
and $\widetilde{\qjm}\subset P+k\ell(P)\nsubseteq 3P$ for some $k\in\zd$,
where $P+k\ell(P):=\{x+k\ell(P):x\in P\}$.
Then $|k|\geq 2$ and $1+\ell(\widetilde{\qjm})^{-1}|x_{\widetilde{\qjm}}-x_{\qjm}|\sim 2^i|k|$.
Note that
\begin{align*}
    &\sum_{\substack{\ell(\qjm)=2^{-i}\ell(P)\\\qjm\subset P}}
    [|\qjm|^{-\frac{s}{d}+\frac{1}{p}}(\mu_{\min(p,q),\delta}^\ast)_{\qjm}]^p\\
    &\qquad\sim\sum_{\substack{\ell(\qjm)=2^{-i}\ell(P)\\\qjm\subset P}}
    |\qjm|^{p(-\frac{s}{d}+\frac{1}{p})}\left(
    \sum_{\substack{k\in\zd\\|k|\geq2}}\sum_{\substack{\ell(\widetilde{\qjm})=2^{-i}\ell(P)\\\widetilde{\qjm}\subset P+k\ell(P)}}
    \frac{|\mu_{\widetilde{\qjm}}|^{\min(p,q)}}{(1+\ell(\widetilde{\qjm})^{-1}|x_{\widetilde{\qjm}}-x_{\qjm}|)^\delta}\right)^\frac{p}{\min(p,q)}\\
    &\qquad\sim \sum_{\substack{\ell(\qjm)=2^{-i}\ell(P)\\\qjm\subset P}}
    |\qjm|^{p(-\frac{s}{d}+\frac{1}{p})}\left(
    \sum_{\substack{k\in\zd\\|k|\geq2}}\sum_{\substack{\ell(\widetilde{\qjm})=2^{-i}\ell(P)\\\widetilde{\qjm}\subset P+k\ell(P)}}
    2^{-i \delta}|k|^{-\delta}|\lz_{\widetilde{\qjm}}|^{\min(p,q)}
    \right)^\frac{p}{\min(p,q)}\\
    &\qquad\sim 2^{id-i\delta\frac{p}{\min(p,q)}} \left(
    \sum_{\substack{k\in\zd\\|k|\geq2}}|k|^{-\delta}
    \sum_{\substack{\ell(\widetilde{\qjm})=2^{-i}\ell(P)\\\widetilde{\qjm}\subset P+k\ell(P)}}
    (|\widetilde{\qjm}|^{-\frac{s}{d}+\frac{1}{p}}|\lz_{\widetilde{\qjm}}|)^{\min(p,q)}
    \right)^\frac{p}{\min(p,q)}.
\end{align*}

Thus,
\begin{align*}
\begin{split}
    {\mathrm{J}_P}&:=\frac{1}{\varphi(\ell(P))}\left\{\sum_{i=-\min(j_P,0)}^\infty
    \left[\sum_{\substack{\ell(\qjm)=2^{-i}\ell(P)\\\qjm\subset P}}
    [|\qjm|^{-\frac{s}{d}+\frac{1}{p}}(\mu_{\min(p,q),\delta}^\ast)_{\qjm}]^p
    \right]^{\frac{q}{p}}\right\}^{\frac{1}{q}}\\
    &\lesssim\frac{1}{\varphi(\ell(P))}\left\{\sum_{i=-\min(j_P,0)}^\infty 2^{{id\frac{q}{p}}-i\delta\frac{q}{\min(p,q)}}\left[
    \sum_{\substack{k\in\zd\\|k|\geq2}}|k|^{-\delta}\right.\right.\\
    &\qquad\qquad\qquad\left.\left.\times\sum_{\substack{\ell(\widetilde{\qjm})=2^{-i}\ell(P)\\\widetilde{\qjm}\subset P+k\ell(P)}}(|\widetilde{\qjm}|^{-\frac{s}{d}+\frac{1}{p}}|\lz_{\widetilde{\qjm}}|)^{\min(p,q)}\right]^{\frac{q}{\min(p,q)}}\right\}^{\frac{1}{q}}.
\end{split}
\end{align*}
When $p\leq q$, by $\delta>d$, we have
\begin{align*}
    \mathrm{J}_P\lesssim\|\lz\mid b_{p,q}^{s,\varphi}(\rd)\|\left\{\sum_{i=-\min(j_P,0)}^\infty2^{-i\frac{q}{p}(\delta-d)}\left(\sum_{\substack{k\in\zd\\|k|\geq2}}|k|^{-\delta}\right)^{\frac{q}{p}}\right\}^{\frac{1}{q}}\lesssim\|\lz\mid b_{p,q}^{s,\varphi}(\rd)\|;
\end{align*}
when $p>q$, by H\"older's inequality and $\delta>d$, we have
\begin{align*}
    {\mathrm{J}_P}&\lesssim\frac{1}{\varphi(\ell(P))}\left\{
    \sum_{i=-\min(j_P,0)}^\infty 2^{id\frac{q}{p}-i\delta}\left[
    \sum_{\substack{k\in\zd\\|k|\geq2}}|k|^{-\delta}\sum_{\substack{\ell(\widetilde{\qjm})=2^{-i}\ell(P)\\\widetilde{\qjm}\subset P+k\ell(P)}}(|\widetilde{\qjm}|^{-\frac{s}{d}+\frac{1}{p}}|\lz_{\widetilde{\qjm}}|)^q
    \right]\right\}^{\frac{1}{q}}\\
    &\lesssim\frac{1}{\varphi(\ell(P))}\left\{
    \sum_{i=-\min(j_P,0)}^\infty 2^{id\frac{q}{p}-i\delta}
    \sum_{\substack{k\in\zd\\|k|\geq2}}|k|^{-\delta}\right.\\
    &\qquad\qquad\qquad\left.
    \times\left[\sum_{\substack{\ell(\widetilde{\qjm})=2^{-i}\ell(P)\\\widetilde{\qjm}\subset P+k\ell(P)}}
    (|\widetilde{\qjm}|^{-\frac{s}{d}+\frac{1}{p}}|\lz_{\widetilde{\qjm}}|)^p\right]^\frac{q}{p}
    \left(\sum_{\substack{\ell(\widetilde{\qjm})=2^{-i}\ell(P)\\\widetilde{\qjm}\subset P+k\ell(P)}}1\right)^{1-\frac{q}{p}}
    \right\}^{\frac{1}{q}}\\
    &\lesssim\|\lz\mid b_{p,q}^{s,\varphi}(\rd)\|\left\{\sum_{i=-\min(j_P,0)}^\infty2^{-i(\delta-d)}\left(\sum_{\substack{k\in\zd\\|k|\geq2}}|k|^{-\delta}\right)\right\}^{\frac{1}{q}}\\
    &\lesssim\|\lz\mid b_{p,q}^{s,\varphi}(\rd)\|.
\end{align*}
Therefore, by \eqref{eel2},
$$
\|\lz_{\min(p,q),\delta}^\ast\mid b_{p,q}^{s,\varphi}(\rd)\|\lesssim\sup_{P\in\mq}(\mathrm{I}_P+\mathrm{J}_P)\lesssim\|\lz\mid b_{p,q}^{s,\varphi}(\rd)\|.
$$

Finally, we show that $\widetilde{\mathrm{J}}_P\lesssim\|\lz\mid f_{p,q}^{s,\varphi}(\rd)\|$, where $$\widetilde{\mathrm{J}}_P:=\frac{1}{\varphi(\ell(P))}\left\{\int_P\left[\sum_{j=\jjp}^\infty
\sum_{\substack{m\in\zd\\\qjm\subset P}}(2^{j s}|(\mu^\ast_{\min(p,q),\delta})_{\qjm}|\chi_{\qjm}(x))^q\right]^{\frac{p}{q}}\,\dint x\right\}^{\frac{1}{p}}.$$
For any $i\in\no$, $k\in\no^d$ and dyadic cube $P$, let
\begin{align*}
    A(i,k,P):=\{\widetilde{\qjm}\in\mq:\,\ell(\widetilde{\qjm})=2^{-i}\ell(P),\,\widetilde{\qjm}\subset P+k\ell(P)\nsubseteq3P\}.
\end{align*}
Let $a:=\frac{2d}{d+\delta}\min(p,q)$. Then $0<a<\min(p,q)$. By monotonicity of $\ell^q$-norm in $q$ and $1+\ell(\widetilde{\qjm})^{-1}|x_{\qjm}-x_{\widetilde{\qjm}}|\sim 2^i|k|$ for any $\qjm\subset P$ and $\widetilde{\qjm}\in A(i,k,P)$, we obtain that for all $x\in P$,
\begin{align*}
    &\sum_{\widetilde{\qjm}\in A(i,k,P)}\frac{(|\widetilde{\qjm}|^{-\frac{s}{d}}|\lz_{\widetilde{\qjm}}|)^{\min(p,q)}}{(1+(\ell(\widetilde{\qjm}))^{-1}
    |x_Q-x_{\widetilde{\qjm}}|)^\delta}\\
    &\qquad\lesssim2^{-i\delta}|k|^{-\delta}\sum_{\widetilde{\qjm}\in A(i,k,P)}(|\widetilde{\qjm}|^{-\frac{s}{d}}|\lz_{\widetilde{\qjm}}|)^{\min(p,q)}\\
    &\qquad\lesssim2^{-i\delta}|k|^{-\delta}\left(\sum_{\widetilde{\qjm}\in A(i,k,P)}(|\widetilde{\qjm}|^{-\frac{s}{d}}|\lz_{\widetilde{\qjm}}|)^a\right)^{\frac{\min(p,q)}{a}}\\
    &\qquad\lesssim2^{-i\delta}2^{id\frac{\min(p,q)}{a}}|k|^{-\delta}\left(\int_{\widetilde{\qjm}}\sum_{\widetilde{\qjm}\in A(i,k,P)}(|\widetilde{\qjm}|^{-\frac{s}{d}}|\lz_{\widetilde{\qjm}}|)^a\chi_{\widetilde{\qjm}}(x+k\ell(P))\dint x\right)^{\frac{\min(p,q)}{a}}\\
    &\qquad\lesssim2^{-i\delta+id\frac{\min(p,q)}{a}}|k|^{-\delta}\left[\mathcal{M}_{\mathrm{HL}}\left(
    \sum_{\substack{\ell(\widetilde{\qjm})=2^{-i}\ell(P)\\\widetilde{\qjm}\subset P+k\ell(P)}}(|\widetilde{\qjm}|^{-\frac{s}{d}}|\lz_{\widetilde{\qjm}}|\chi_{\widetilde{\qjm}})^a
    \right)(x+k\ell(P))\right]^{\frac{\min(p,q)}{a}}.
\end{align*}
Then by Minkowski's inequality, Fefferman-Stein's vector-valued inequality and H\"older's inequality, we have
\begin{align*}
    {\widetilde{\mathrm{J}}_P}&\lesssim\frac{1}{\varphi(\ell(P))}\left\{\int_P\left[
    \sum_{i=-\min(j_P,0)}^\infty\left(
    \sum_{\substack{k\in\zd\\|k|\geq2}}2^{-i\delta+id\frac{\min(p,q)}{a}}|k|^{-\delta}\right.\right.\right.\\
    &\,\,\,\,\,\,\,\,\,\left.\left.\left.\times\left[\mathcal{M}_{\mathrm{HL}}\left(
    \sum_{\substack{\ell(\widetilde{\qjm})=2^{-i}\ell(P)\\\widetilde{\qjm}\subset P+k\ell(P)}}(|\widetilde{\qjm}|^{-\frac{s}{d}}|\lz_{\widetilde{\qjm}}|\chi_{\widetilde{\qjm}})^a
    \right)(x+k\ell(P))\right]^{\frac{\min(p,q)}{a}}
    \right)^{\frac{q}{\min(p,q)}}
    \right]^{\frac{p}{q}}
    \,\dint x\right\}^{\frac{1}{p}}\\
    &\lesssim\|\lz\mid f_{p,q}^{s,\varphi}(\rd)\|.
\end{align*}
Therefore, by \eqref{eel2} again,
$$
\|\lz_{\min(p,q),\delta}^\ast\mid f_{p,q}^{s,\varphi}(\rd)\|\lesssim\sup_{P\in\mq}(\mathrm{\widetilde{I}}_P+\mathrm{\widetilde{J}}_P)\lesssim\|\lz\mid f_{p,q}^{s,\varphi}(\rd)\|,
$$
which completes the proof of Lemma~\ref{el2}.
\end{proof}

\begin{lemma}{\rm{(\cite[Corollary 4.6]{nns16})}}\label{fmd}
    Let $M>0$ be arbitrary. Let $K\in\no$ and $L\in\no\cup\{-1\}$. Let an atom $a_{\qjm}$ be given, supported near
    $\qjm$ with $j\in\no,\,m\in\zd$. Assume that $\nu\in\nat$, $j\in\no$ and $m\in\zd$. Then for $x\in\rd$,
    \begin{align*}
        |\mathcal{F}^{-1}(\theta_\nu\mathcal{F}a_{\qjm})(x)|\lesssim
        \begin{cases}
            2^{-(\nu-j)K}\mhl^{(\frac{d}{M})}(\kjm)(x),&\qquad\nu\geq j,\\[1ex]
            2^{(\nu-j)(L+1+d-M)}\mhl^{(\frac{d}{M})}(\kjm)(x),&\qquad\nu< j.
        \end{cases}
    \end{align*}
    In particular, if
    \begin{align}\label{dac}
        \kappa:=\min(L+1+d-M+s,\, K-s),
    \end{align}
    then
    \begin{align*}
        2^{\nu s}\left|\mathcal{F}^{-1}\left(\theta_\nu\mathcal{F}\sum_{m\in\zd}\lzjm a_{\qjm}\right)(x)\right|
        \lesssim 2^{-|\nu-j|\delta}\mhl^{(\frac{d}{M})}\left(\sum_{m\in\zd} 2^{j s}\lzjm a_{\qjm}\right)(x),\quad x\in\rd,
    \end{align*}
    where, here and hereafter, $\mhl^{(\eta)}$ denotes the powered Hardy-Littlewood maximal operator, namely,
    $$\mhl^{(\eta)}(g)(x):=\sup_{Q\in\mathcal{Q}_x(\rd)}\left(\frac{1}{|Q|}\int_Q|g(y)|^\eta\,\dint y
    \right)^\frac{1}{\eta}$$
    for all measurable functions $g$ and $0<\eta<\infty$.
\end{lemma}

Now we are ready to prove Theorem~\ref{adbf}, which is inspired by the proofs of \cite[Theorem 4.1]{fj90} and \cite[Theorem 4.5]{nns16}.

\begin{proof}[Proof of Theorem~\ref{adbf}]
\emph{Step 1.} We first prove (i).
Let $f\in\Att(\rd)$, and let $\theta$ and $\upsilon$ be as in \eqref{crf2}. By Lemma~\ref{crf},
we write $f=\sum_{j\in\no}\sum_{m\in\zd}\lzjm\mathcal{F}^{-1}\upsilon_{\qjm}$ in $\sdd$, where
\begin{align*}
    \lz:=\{\lzjm\}_{j\in\no,m\in\zd}:=\{\langle f,\,\mathcal{F}^{-1}\theta_{\qjm}\rangle\}_{j\in\no,m\in\zd}
\end{align*}
satisfies $\|\lz\mid \att(\rd)\|\lesssim\|f\mid \Att(\rd)\|$.
Choose $\vartheta$ which satisfies $\mathcal{F}^{-1}\vartheta\in\sd$,
$\supp \mathcal{F}^{-1}\vartheta\subseteq\{x\in\rd:\,|x|\leq 1\}$,
$\int_{\rd}x^{\beta}\mathcal{F}^{-1}\vartheta(x)\,\dint x=0$ if $|\beta|\leq L$ when $L\geq 0$,
and $|\vartheta(\xi)|\geq C>0$ if $\frac{1}{2}\leq |\xi|\leq 2$ (see \cite[p.783]{fj85} for a detailed construction of $\vartheta$).
By \eqref{3e1} and \eqref{3e2}, there exists a function $\eta$ such that  $\mathcal{F}^{-1}\eta\in\sd$ and
$\mathcal{F}^{-1}\upsilon=\mathcal{F}^{-1}\vartheta\ast\mathcal{F}^{-1}\eta$. Setting
\begin{align*}
    g_k:=\int_{Q_{0,k}}\mathcal{F}^{-1}\vartheta(\cdot-y)\mathcal{F}^{-1}\eta(y)\,\dint y
\end{align*}
for all $k\in\zd$, we have $\mathcal{F}^{-1}\upsilon=\sum_{k\in\zd}g_k$, and, hence, for all $j\in\no$, $m\in\zd$, we set
\begin{align*}
    \mathcal{F}^{-1}\upsilon_{\qjm}:=\sum_{k\in\zd}g_k(2^{j}x-m).
\end{align*}
Note that $\int_{\rd}x^\beta g_k(x)\,\dint x=0$, if $|\beta|\leq L$ when $L\geq 0$ and
$|\Dd^\alpha g_k(x)|\leq C_1 (1+|k|)^{-S}$ for any $S>0$,
where $C_1$ only depends on $S$ and $\alpha$.

For $Q_{j,m}$ with $j\in\no,\,m\in\zd$, set $r_{\qjm}:=C(\lz^\ast_{\min(p,q),\delta})_{\qjm}$ and
\begin{align*}
    a_{\qjm}:=\sum_{m\in\zd}\lz_{\qjm}\frac{g_{k-m}(2^jx-m)}{r_{\qjm}},
\end{align*}
where $C$ will be determined later. From the above representations of $f$ and $\mathcal{F}^{-1}\upsilon$, we deduce that
\begin{align*}
    f=\sum_{j\in\no}\sum_{m\in\zd}r_{\qjm} a_{\qjm} \qquad\text{in}\qquad \sdd.
\end{align*}
Let $S\geq \delta$. Then
\begin{align*}
    |\Dd^\alpha a_{\qjm}(x)|&\leq C_2 2^{j |\alpha|}
    \sum_{R\in\mq, \ \ell(R)=\ell(\qjm)}\frac{|\lz_{R}|(1+\ell(\qjm)^{-1}|x_{R}-x_{\qjm}|)^{-S}}{C (\lz_{\min(p,q),\delta}^\ast)_{\qjm}}\\
    &\leq \frac {C_3 }{C}2^{j |\alpha|},
\end{align*}
where the first inequality follows from the above estimate of $\Dd^\alpha g_k$ and
the last inequality follows from (reverse) H\"older's inequality if $\min(p,q)>1$,
or the monotonicity of $\ell^q$-norm in $q$ if $\min(p,q)\leq 1$.
Taking $C$ large enough yields \eqref{ac2}. Clearly \eqref{ac3} holds. Thus, each $a_{\qjm}$ is an atom for $\Att(\rd)$.
Moreover, let $\delta>d$, by Lemma~\ref{el2}, we see that
\begin{align*}
    \|r\mid \btts(\rd)\|\lesssim\|\lz\mid\btts(\rd)\|\lesssim\|f\mid\btt(\rd)\|,
\end{align*}
which finishes the proof of (i) of Theorem~\ref{adbf}.

\emph{Step 2.} We now prove (ii). Let $\{a_{\qjm}\}_{j\in\no,m\in\zd}$ be a family of $(K,L,c)$-atoms and
$\lz=\{\lzjm\}_{j\in\no,m\in\zd}\in\att(\rd)$.
We first assume that there exists $M\gg 1$ such that $\lzjm=0$ if $j\geq M$. This implies
\begin{align*}
    f:=\sum_{j=0}^\infty\sum_{m\in\zd}\lzjm \ajm
\end{align*}
converges in $\sdd$.

\emph{Step 2.1.} Choose $M>0$ such that
\begin{align}\label{4.13}
\frac{d}{\min(1,p,q)}<M<L+d+1+s.
\end{align}
Then $\kappa$ given by \eqref{dac} is positive.  By Lemma~\ref{fmd}, (i) of Proposition~\ref{ggl}, \eqref{maxp1} and \eqref{4.13},
\begin{align*}
    &\frac{1}{\varphi(\ell(P))}\left\{\sum_{\nu=\jjp}^\infty\left[\int_P\left(\sum_{j=0}^\infty
    2^{\nu s}\left|\mathcal{F}^{-1}\left(\theta_\nu\mathcal{F}\sum_{m\in\zd}\lzjm \ajm\right)(x)\right|\right)^p\,\dint x
    \right]^\frac{q}{p}\right\}^\frac{1}{q}\\
    &\quad\lesssim\frac{1}{\varphi(\ell(P))}\left\{\sum_{\nu=\jjp}^\infty\left[\int_P\left(\sum_{j=0}^\infty
    2^{-|\nu-j|\delta}\mhl^{(\frac{d}{M})}\left(\sum_{m\in\zd}2^{j s}\lzjm \ajm\right)(x)\right)^p\,\dint x
    \right]^\frac{q}{p}\right\}^\frac{1}{q}\\
    &\quad\lesssim\frac{1}{\varphi(\ell(P))}\left\{\sum_{j=\jjp}^\infty
    \left[\int_P\left(\mhl^{(\frac{d}{M})}\left(\sum_{m\in\zd}2^{j s}\lzjm \ajm\right)(x)\right)^p\,\dint x
    \right]^\frac{q}{p}\right\}^\frac{1}{q}\\
    &\quad\lesssim\frac{1}{\varphi(\ell(P))}\left\{\sum_{j=\jjp}^\infty
    \left[\int_P\left(\sum_{m\in\zd}2^{j s}\lzjm \ajm\right)^p\,\dint x
    \right]^\frac{q}{p}\right\}^\frac{1}{q}.
\end{align*}
Thus, $\|f\mid \btt(\rd)\|\lesssim\|\lz\mid\btts(\rd)\|$.

Similarly, by \eqref{4.13}, Lemma~\ref{fmd}, (ii) of Proposition~\ref{ggl} and \eqref{maxp2}, we have
\begin{align*}
    &\frac{1}{\varphi(\ell(P))}\left\{\int_P\left[
    \sum_{\nu=\jjp}^\infty\left(\sum_{j=0}^\infty
    2^{\nu s}\left|\mathcal{F}^{-1}\left(\theta_\nu\mathcal{F}\sum_{m\in\zd}\lzjm \ajm\right)(x)\right|\right)^q
    \right]^\frac{p}{q}\,\dint x\right\}^\frac{1}{p}\\
    &\quad\lesssim\frac{1}{\varphi(\ell(P))}\left\{\int_P\left[\sum_{\nu=\jjp}^\infty\left(\sum_{j=0}^\infty
    2^{-|\nu-j|\delta}\mhl^{(\frac{d}{M})}\left(\sum_{m\in\zd}2^{j s}\lzjm \ajm\right)(x)\right)^q
    \right]^\frac{p}{q}\,\dint x\right\}^\frac{1}{p}\\
    &\quad\lesssim\frac{1}{\varphi(\ell(P))}\left\{\int_P
    \left[\sum_{j=\jjp}^\infty\left(\mhl^{(\frac{d}{M})}\left(\sum_{m\in\zd}2^{j s}\lzjm \ajm\right)(x)\right)^q
    \right]^\frac{p}{q}\,\dint x\right\}^\frac{1}{p}\\
    &\quad\lesssim\frac{1}{\varphi(\ell(P))}\left\{\int_P
    \left[\sum_{j=\jjp}^\infty\left(\sum_{m\in\zd}2^{j s}\lzjm \ajm\right)^q
    \right]^\frac{p}{q}\,\dint x\right\}^\frac{1}{p},
\end{align*}
which yields $\|f\mid \ftt(\rd)\|\lesssim\|\lz\mid\ftts(\rd)\|$.
Hence, we show that (ii) of Theorem~\ref{adbf} is proved for $\lz$ satisfying that there exists $M\gg1$
such that $\lzjm=0$ if $j\geq M$.

\emph{Step 2.2} Next we remove the assumption in Step 2.1. Set
\begin{align*}
    f_j:=\sum_{m\in\zd}\lz_{\qjm}a_{\qjm},\qquad j\in\no.
\end{align*}
Choose $\rho>0$ such that
\begin{align*}
    L\geq \max(-1,\whole{\sigma_p-s+\rho}),\quad\text{for}\,\,\, B\text{-spaces}
    \qquad\text{and}\qquad
    L\geq \max(-1,\whole{\sigma_{p,q}-s+\rho}),\quad\text{for}\,\,\, F\text{-spaces}.
\end{align*}
Then according to what we have proved, we obtain
\begin{align*}
    \|f_j\mid A_{p,q}^{s-\rho,\varphi}(\rd)\|\lesssim 2^{-\rho j}\|\lz\mid a_{p,q}^{s,\varphi}(\rd)\|\qquad\text{and}\qquad
    \left\|\sum_{j=1}^M f_j\mid \Att(\rd)\right\|\leq C_4 \|\lz\mid a_{p,q}^{s,\varphi}(\rd)\|,
\end{align*}
where $C_4$ is independent of $M$. Therefore, $f=\sum_{j\in\no}f_j$ converges in $A_{p,q}^{s-\rho,\varphi}(\rd)$ and hence in $\sdd$. By letting $M\rightarrow \infty$ and the Fatou property, we complete the proof of Theorem~\ref{adbf}.
\end{proof}


\vfill\bigskip~

{\small
\noindent
Dorothee D. Haroske\\
Institute of Mathematics \\
Friedrich Schiller University Jena\\
07737 Jena\\
Germany\\
{\tt dorothee.haroske@uni-jena.de}\\[4ex]
Zhen Liu\\
Institute of Mathematics \\
Friedrich Schiller University Jena\\
07737 Jena\\
Germany\\
{\tt zhen.liu@uni-jena.de}\\[4ex]
}

\end{document}